\documentclass[10pt]{amsart}

\usepackage{amsmath,  amsthm, enumerate, amssymb, mathtools, csquotes}
\usepackage[provide=*,english]{babel}

\usepackage{tikz-cd, tikz}

\usepackage{mathrsfs}

\usepackage{hyperref}

\usepackage{graphicx}

\newcommand{\Z}{\mathbb{Z}}
\newcommand{\R}{\mathbb{R}}

\newcommand{\Q}{\mathbb{Q}}
\newcommand{\C}{\mathbb{C}}

\newcommand{\Or}{\mathcal{O}}

\newcommand{\F}{\mathbb F}

\newcommand{\p}{\mathfrak p}

\newcommand{\h}{\mathfrak h}

\DeclareMathOperator{\PSL}{PSL}
\DeclareMathOperator{\SL}{SL}
\DeclareMathOperator{\GL}{GL}

\DeclareMathOperator{\Free}{Free}
\DeclareMathOperator{\Forget}{Forget}
\DeclareMathOperator{\Ob}{Ob}
\DeclareMathOperator{\Mor}{Mor}
\DeclareMathOperator{\Sym}{Sym}
\DeclareMathOperator{\Hom}{Hom}

\DeclareMathOperator{\im}{Im}
\DeclareMathOperator{\codim}{codim}

\newcommand{\op}{\textrm{op}}
\newcommand{\id}{\textrm{id}}

\newcommand{\Fs}{\mathscr{F}}
\newcommand{\G}{\mathscr{G}}

\numberwithin{equation}{section}

\theoremstyle{plain}
\newtheorem{thm}[equation]{Theorem}
\newtheorem{lem}[equation]{Lemma}
\newtheorem{prop}[equation]{Proposition}
\newtheorem{cor}[equation]{Corollary}

\theoremstyle{definition}
\newtheorem{prob}[equation]{Problem}
\newtheorem{defn}[equation]{Definition}
\newtheorem{nota}[equation]{Notation}

\newtheorem{algo}{Algorithm}[subsection]

\theoremstyle{remark}
\newtheorem{rem}[equation]{Remark}
\newtheorem{ex}{Example}

\begin{document}
\title[Computing the cohomology of Shimura curves in quasi-linear time]{Computing the cohomology of Shimura curves\\in quasi-linear time}

\author{Rayane Baït$^{\ast}$ and Aurel Page}
\address{Univ. Bordeaux, CNRS, Inria, Bordeaux INP, IMB, UMR 5251, F-33400 Talence, France}
\thanks{$^\ast$Correspondence: \url{rayane.bait@math.u-bordeaux.fr}}
\date{}



\begin{abstract}
	Computing spaces of modular and automorphic forms is an important
problem in algorithmic number theory, with in particular Diophantine
applications to generalised Fermat and other equations. The case of Shimura
curves was studied by Greenberg and Voight by cohomological methods, allowing
them to reduce the problem to linear algebra. Relying on work of Imbert, we leverage
topological techniques to obtain linear systems so structured that they can be
solved in time quasi-linear in the genus.
\end{abstract}

\maketitle
\section{Introduction}
The computation of spaces of modular and automorphic forms is a major theme in
the algorithmic number theory community, with a long tradition of supporting
theoretical developments and testing conjectures, efforts of systematic
tabulation ongoing with the constant growth of
extensive databases such as the LMFDB \cite{lmfdb}, and applications to
Diophantine problems generalising the successful proof of Fermat's last theorem.
Unlocking new types of automorphic forms is interesting, but improving the
efficiency of known cases is equally important.
For instance to realize Darmon's program
on generalised Fermat equations,
one needs to compute various spaces of automorphic forms,
but the levels that appear can be so large that the spaces become difficult
to compute (see for instance~\cite[Section 8]{BillereyChenDieulefaitFreitas}).

Spaces of classical modular forms can be efficiently computed using the method of
modular symbols.
For Hilbert modular forms, Dembélé, Donnelly, Greenberg and Voight
\cite{dembdonn,greenbergvoightHilbert,voightarbitrary,dembelevoightHilbert} gave
practical algorithms to obtain the 
Hecke-module structure of spaces of Hilbert modular forms.
In the odd degree case, the so called ``indefinite method''
requires the computation of the cohomology~$H^1(\Gamma, V)$ for 
some congruence arithmetic Fuchsian groups $\Gamma$ and certain modules $V$;
this is achieved by computing a presentation of~$\Gamma$ from a fundamental
domain, then solving a linear system of size~$O(\mu\dim V)$ where~$\mu$ denotes the
co-area of~$\Gamma$. While practical, this method
is far from the quasi-linear complexity of modular symbols.

In this work, we describe and implement an algorithm to compute a basis of~$H^1(\Gamma,V)$ whose
complexity depends quasi-linearly on~$\mu$.
The main idea is that instead of using the presentation of~$\Gamma$ directly
derived from the fundamental domain, we can leverage the theory of surfaces to
obtain a presentation of a special form. This allows us to compute the
cohomology group~$H^1(\Gamma,V)$ from a linear system that is so structured that
it can be solved in quasi-linear time.

To achieve this, we revisit the work of M. Imbert
\cite{Imbert1,Imbert2} in the framework of groupoids and graph
embeddings, to turn his method into 
rigorous and explicit algorithms. By crucially using straight line programs and
with sufficient care, this leads to algorithms
with good complexity. For simplicity, in this introduction we state our results
assuming that~$\Gamma$ is a congruence arithmetic Fuchsian group.
In Section~\ref{sec:algos}, we prove the following result (see
Theorem~\ref{thm:algopres}).
\begin{thm}\label{thm:specialpres}
	There exists an algorithm that, given a fundamental polygon of a co-compact Fuchsian
	group $\Gamma$ with its side pairing and an integer~$m\ge 0$, computes an
	$m$-handle presentation of $\Gamma$ of the form
	\begin{align*}
		\langle a_1,b_1,\ldots, a_m,b_m,&e_{2m+1}\ldots, e_{2g},\gamma_1,\ldots,
        \gamma_f\mid w\cdot \prod_{k=1}^f\gamma_k,\gamma_j^{\nu_j}\rangle
	\end{align*}
    where $w=\prod_{j=1}^m[a_j,b_j]w'$
    with~$w'$ a product of the~$e_j^{\pm 1}$, each
    appearing exactly once,
    and runs in time $\tilde O((m+1)\mu)$, where~$\mu$ denotes the co-area of~$\Gamma$. 
\end{thm}

We then exploit these special presentations for cohomological computations, the
intuition being that~$\Gamma$ is almost a free product of cyclic groups. A
first idea is to use a geometric presentation ($m=g$ above), but
Theorem~\ref{thm:specialpres} only gives quadratic complexity for this case.
On the other hand, a one-word presentation ($m=0$) does not yield a sufficiently
structured system. A key insight is that a one-handle presentation ($m=1$) is
sufficient.
More precisely, in Section \ref{sec:cohomology} we prove the following theorem
(see Proposition~\ref{prop:computeH1} for a more precise statement).

\begin{thm}\label{thm:introcomputeH1}
	There exists an algorithm that, given a fundamental polygon for $\Gamma$
    with its side pairing
    and an integer~$k\ge 0$, computes an 
	explicit basis of 
	\[
      H^1(\Gamma, \Sym^k(\C^2))
    \]
    in time $\tilde O(\mu (k+1)^3)$.
\end{thm}

By ``computing an explicit basis'' of this finite-dimensional vector space, we
mean more than computing its dimension; see Definition~\ref{def:cohom} for a
precise definition: the basis is computed in a form suitable for cohomological
operations such as Hecke operators.
We implemented our algorithms in PARI/GP~\cite{PARI2}, so far in the case~$k=0$,
and compared them to the existing Magma~\cite{magma} implementation.
The experimental results are presented in Section~\ref{sec:implem}.

\subsection*{Notations and conventions}
We write~$\langle g_1,\dots, g_n\rangle$ for the group generated by
elements~$g_1,\dots, g_n$ and we write~$\langle g_1,\dots, g_n \mid r_1,\dots,r_m\rangle$
for the finitely presented group with generators~$g_i$ and relations~$r_j$.
When a group~$G$ acts on a space~$V$, for $g\in G$ and $v\in V$
we denote by $g.v$ the image of $v$ by $g$. For a subset $W$ of
$V$ we denote by $g.W$ the set $\{g.w| w\in W\}$ and by $G.W$
the set $\cup_{g\in G}g.W$. We denote by~$V^G$ the set of fixed points
of~$G$; for~$g\in G$, we abbreviate~$V^{\langle g  \rangle} = V^g$. 
Let~$\omega$ be an admissible exponent for linear algebra, i.e. such that matrix
multiplication of dimension~$n\times n$ can be performed using~$O(n^\omega)$
field operations. Lastly, given a topological space $X$ and a subset 
$S$ of $X$, we denote by $\mathring S$ (resp. $\partial S$) the
interior (resp. boundary) of $S$ in $X$.
We use the soft O notation: $g = \tilde O(f)$ if there exists~$a>0$ such that~$g =
O(f \log(2+|f|)^a)$.

\subsection*{Acknowledgements}
The authors express their gratitude towards the referees, whose remarks helped
improve the exposition of the paper.
We also thank Vincent Delecroix and Jean-Marc Couveignes for enriching
discussions around this project.
This research was funded by the ANR AGDE project (ANR-20-CE40-0010) and by the
European Union under Horizon Europe research and innovation programme, Grant
Agreement 101169527 (COGENT).

\section{Preliminaries}
\subsection{Presentations from group actions}\label{subsec:presfromga}
Let $\Gamma$ be a topological group and $X$ be a Hausdorff locally compact topological
space. We say a left action of $\Gamma$ on $X$ is \emph{properly discontinuous} if
$\Gamma$ is discrete and for every $x\in X$ there exists some neighborhood $V$ of $x$
such that the set $\{g\in \Gamma\mid gV\cap V\ne \emptyset\}$ is finite. Further, the action
is said to be \emph{totally discontinuous} if it is also free. The next
theorem shows how to reduce the computation of presentations for groups
acting properly discontinuously to the computation of presentations for 
fundamental groups. 
\begin{thm}\label{thm:devissage}
	 Given a totally discontinuous action of $\Gamma$ on $X$,
	 the projection map 
	$p\colon X\to \Gamma\backslash X$ is a normal covering map.
    Let $\Gamma(p)$ be
	its automorphism group; then $\Gamma(p)\simeq \Gamma$, and the sequence
	\[
      1\to p_*(\pi_1(X,x))\to \pi_1(\Gamma\backslash X,p(x))\to \Gamma(p)\to 1
    \]
	is exact.
\end{thm}
\begin{proof}
	A proof is given in \cite[Proposition 1.39; Proposition 1.40]{hatcher}.
\end{proof}
\begin{rem}\label{rem:globalmonodromy}
	The surjective map 
	$\pi_1(\Gamma\backslash X, \Gamma.x)\to \Gamma$ is defined
	as follows: every loop $\gamma$ at $p(x)$ lifts 
	to a unique path $\tilde \gamma$ such that $\tilde \gamma(1)=x$ and there
	is some $g_\gamma\in \Gamma$ such that $\tilde \gamma(0)=g_\gamma.x$.
    We then associate~$\gamma \mapsto g_\gamma$.
\end{rem}
We now assume $X$ is path connected and we fix a properly discontinuous action of $\Gamma$ on $X$. Let $Y$ be the set of points
of $X$ with non trivial stabilizer. Then
\[ X-Y\to \Gamma\backslash (X-Y) \]
is a covering map and $\Gamma$ acts totally discontinuously on
$X-Y$.
In addition, if we assume that $Y$ is discrete 
and $X-Y$ stays path connected then the inclusion map 
$X-Y\subset X$ induces a surjective morphism
$\pi_1(X-Y,x)\to \pi_1(X,x)$
with kernel generated by conjugacy classes of simple loops around
points of $Y$. These loops are usually easy to describe.

When $X=\h$ is the upper-half plane and $\Gamma$ is a Fuchsian
group (see Section~\ref{subsec:fuchsian}), the quotient space has the structure of a good
compact complex $1$-orbifold \cite[Chapter 34.8.14, Chapter 3.6]{quatbook, katok} so
in particular of a compact orientable topological surface. In this case Brahana 
\cite{brahana}
gave combinatorial tools to compute various presentations of 
their fundamental groups that we introduce in Section 
\ref{subsec:fgge}.

\begin{rem}\label{rem:orblooporder}
 In \cite[Chapter 13]{Thurston} Thurston defines the
 \emph{orbifold} fundamental group $\pi_1^{orb}$ as the
 automorphism group of
the covering $X-Y\to \Gamma\backslash (X-Y)$ which is $\Gamma$ by
the last proposition. 
In the simplest case of the covering $D-\{0\}\to (\Z/n\Z)\backslash (D-\{0\})$
where~$D$ is the unit disk,
the simple loop around $0$ maps to $n$ times the simple loop
around $0$ in the quotient. In this case Theorem \ref{thm:devissage} gives $\pi_1^{orb}((\Z/n\Z)\backslash D,
  1/2)\simeq\langle \gamma \mid \gamma^n\rangle$. 
\end{rem}

\subsection{Fuchsian groups}\label{subsec:fuchsian}
In this section we introduce Fuchsian groups. A standard reference is \cite[Chapters 3-4]{katok}.

	We denote by $\h$ the upper-half plane. We let
$\PSL_2(\R)=\SL_2(\R)/\{\pm 1\}$ be the projective special linear
group over $\R$. It acts properly on $\h$ by homographies and
can be identified with the group of orientation preserving
isometries of $\h$ through this action. 

  A \emph{Fuchsian group} is a discrete subgroup of $\PSL_2(\R)$.
As subgroups of $\PSL_2(\R)$, Fuchsian groups act properly
discontinuously on $\h$ so that we will be able to apply 
Section \ref{subsec:presfromga}. 
For simplicity we will stick to co-compact Fuchsian groups, 
that is, Fuchsian groups $\Gamma\subset \PSL_2(\R)$ such that $\Gamma\backslash \h$ is compact. From now on we fix
$\Gamma\subset \PSL_2(\R)$ a co-compact Fuchsian group.

A powerful tool to study them via their action on $\h$ is the theory of fundamental
domains.
\subsection{Fundamental domains.}\label{subsec:fdom}
We define a \emph{polygon} $P\subset \h$ to be a hyperbolically convex closed domain with
finite
non-zero area such that its boundary is a finite union of maximal
geodesics of $P$.

  Each pair of edges either have empty intersection
  or intersect in a point that we call a vertex. We write
  $E(P)$ (resp. $V(P)$) for the set of edges (resp. vertices) of
  $P$.
  The edges and vertices for any finite union of polygons are defined in the same way. 
  A closed domain $O\subset \h$ in the
  upper-half plane is a \emph{fundamental domain} for $\Gamma$ 
if  $\Gamma.O=\h$ and for all $g\neq 1\in \Gamma$, 
we have $g\mathring O\cap \mathring O=\emptyset$. 

By \cite[Theorem 4.1.1]{katok}, there always exists
a fundamental domain $P$ for $\Gamma$ which is a
polygon. In this case we will call $P$ a fundamental polygon 
for $\Gamma$.

\subsection{Side pairings.}\label{subsec:sidepairings}
	We fix a fundamental polygon $P$ for $\Gamma$. Define a subset
	$R=\{g\in\Gamma\mid g^{-1}.P\cap P\ne \emptyset \}$ of $\Gamma$ associated
    with $P$. For any $g\in R$, write $e=g^{-1}P\cap P$. Then $e$ is 
	a maximal geodesic in $\partial P$.
	When $g.e\ne e$ we call $e$ an edge of $P$. When $g.e=e$, the element~$g$
    has order two 
	and has a fixed point $v$ at the middle of $e$.
    Therefore $e=e_1\cup e_2$ is a union of
	geodesic segments intersecting at $v$ and such that $e_1=g.e_2$. In this case
	instead of $e$ we call $e_1$ and $e_2$ edges of $P$.

	Let $E$ be the set of edges of $P$. For each
	$e\in E$ there is a unique $g_e\in R$ such that 
	$g_e.e\in E$,
	namely the only $g\in \Gamma$ such that 
	$g^{-1}.P\cap P=e$. We write $\bar e=g_e.e$, and with the 
	previous convention we have $\bar e\ne e$ and 
	$\bar{ \bar e} =e$ for all edges $e\in E$. The
map $\sigma_1\colon e\mapsto \bar e$ is then a
fixed-point free involution of $E$. Lastly let 
	$\varphi\colon E\to R$ be the map $e\mapsto g_e^{-1}$. 

	The pair $(\sigma_1, \varphi)$
	is called the \emph{side pairing} of $\Gamma$ attached to $P$.
We also denote by $\sigma_2$ the permutation of the edges of 
$P$ in clockwise orientation and $\sigma_0=\sigma_1\sigma_2^{-1}$.
\begin{thm}[Side pairings generate Fuchsian groups]\label{thm:sidepairingpres}
	The set $R$ is a set of generators for $\Gamma$ and if
	$\{c_v\}_v$ is the set of cycles of $\sigma_0$
	then $\{c_v^{\nu_v}\}$ is a complete set of relations
	for $\Gamma$, where $\nu_v$ is the order of $c_v$ in 
	$\Gamma$.
\end{thm}
\begin{proof}
	\cite[Theorem 3.5.4]{katok}. See also Part \ref{subsec:fgS} and Theorem \ref{thm:devissage}.
\end{proof}
In particular, Fuchsian groups can be specified by
the data of a fundamental domain and a side pairing.

\subsection{Arithmetic Fuchsian groups}
Some general references for this section are \cite{vigneras, katok, quatbook}. 
Let $F$ be a totally real number field of degree $n=[F:\Q]$ over $\Q$ and let $B$ be
a quaternion algebra over $F$, that is a central simple $F$-algebra of dimension $4$.
Any such algebra can be specified by two generators $i,j\in B$ such that $i^2=a\in F$,
$j^2=b\in F$ and $ij=-ji$, we then write $B=\left(\frac{a,b}{F}\right)$.
For every place $v$ of $F$, denote by $F_v$ the completion of $F$ at $v$. The algebra
$B$ is said to split at $v$ if $B\otimes_F F_v\simeq M_2(F_v)$ and ramified 
otherwise.
An \emph{order} $\Or$ of $B$ is a $\Z_F$-lattice of $B$ that is also a subring.
It is \emph{maximal} if it is not properly contained in any other order. Denote by
$\Or^1$ the group of units of reduced norm $1$ in $\Or$. 
Let $r$ be the number of ramified infinite places and $s$ the number of split
infinite places. We have $r+s=n$ and there is an isomorphism
$B\otimes_{\Q} \R\to \mathbb H^r\times M_2(\R)^s$
where $\mathbb H=\left(\frac{-1,-1}{\R}\right)$ 
which induces an embedding 
$\iota\colon \Or^1\to (\mathbb H^1)^r\times \SL_2(\R)^s$.
Denote by $p\colon \mathbb H^r\times M_2(\R)^s\to M_2(\R)^s$ the projection.
The group $p\circ\iota(\Or^1)$ acts diagonally by orientation preserving isometries
on $\h^s$.
The group $p\circ\iota(\Or^1)$ is discrete, and co-compact if~$B$ is a division
algebra \cite[Proposition 38.3.8, Main Theorem 38.4.2]{quatbook}.

When $s=1$, this means that $\Or^1$ is a co-compact
Fuchsian group. We can now define arithmetic Fuchsian groups.

Two groups $H$ and $\Gamma$ are \emph{commensurable} if $H\cap
\Gamma$ has finite index in both $H$ and $\Gamma$. 
An \emph{arithmetic Fuchsian group} is a Fuchsian group which is
commensurable with $p\circ\iota(\Or^1)$ for some $F,B,\Or$ as above. 
\begin{thm}\label{thm:algofdom}
	There exists an algorithm 
	that, given an arithmetic Fuchsian
	group $\Gamma$ identified by arithmetic data, returns
	a fundamental polygon and a side pairing for $\Gamma$,
	heuristically in time $O(\mu^2)$, where $\mu=\mu(\Gamma\backslash \h)$
	is the co-area of~$\Gamma$. 
\end{thm}
\begin{proof}
	\cite[Algorithm 3.1]{rickards}.
\end{proof}
%

By Sections~\ref{thm:sidepairingpres} and \ref{thm:algofdom}, given an
arithmetic Fuchsian group identified by arithmetic data we can
compute a finite set of generators attached to a fundamental
domain and a side pairing identifying our group.
This identification of an arithmetic Fuchsian group is already
very convenient, but the given presentation is not the best one for
computations.

\subsection{Signature of a Fuchsian group.}\label{subsec:sig}
    Let $(C_i)_{i=1,\ldots,f}$ be the conjugacy classes of 
	maximal finite subgroups of $\Gamma$.
	The signature of $\Gamma$ is a tuple 
	$(g;\nu_1,\ldots,\nu_f)$ where
	$g$ is the genus of the surface $\Gamma\backslash \h$ and
	$\nu_i$ is the order of any maximal 
	subgroup in $C_i$ for every~$i$.


\begin{ex}[An arithmetic Fuchsian group with signature (1;3)]\label{ex:congsg13}
	Inputing $F=\Q(\sqrt 8)$, 
	$B=\left(\frac{-\sqrt{2}, -3}{F}\right)$ and 
	$\eta=\Z_F$ in Algorithm 3.1 of \cite{rickards} yields a 
	fundamental polygon as in Figure \ref{fig:fdom} for the
	arithmetic Fuchsian group~$\Gamma$ corresponding to some maximal order.
	It also outputs a side pairing $(\sigma_1,\varphi)$ on the edges 
	$E=\{a,~b,~c,~d,~e,~\bar a,~\bar b,~\bar c,~\bar d,~\bar e\}$
	such that $\sigma_2=(a b c d \bar b e \bar c \bar d \bar e \bar a) $. 
	Then, $R=\varphi(E)$ is a generating family for $\Gamma$
	and computing the cycle decomposition
	\[\sigma_1\circ\sigma_2^{-1}=(a)(b\bar a e)(c\bar b\bar d)(d\bar c\bar e) \]
	gives a presentation 
	\[ \Gamma\simeq \langle a,b,c,d,e\mid a^n,~ba^{-1}e,~cb^{-1}d^{-1}, dc^{-1}e^{-1}\rangle \]
	for $n=|a|_\Gamma$. The algorithm also outputs the signature 
	$(1;3)$ so that $n=3$. Then an explicit
	geometric presentation as in Theorem \ref{thm:specialpres} will be described in 
	Example~\ref{ex:geomprescongsg13}.
\end{ex}

\begin{figure}
	\centering
	\includegraphics[width=0.3\textwidth]{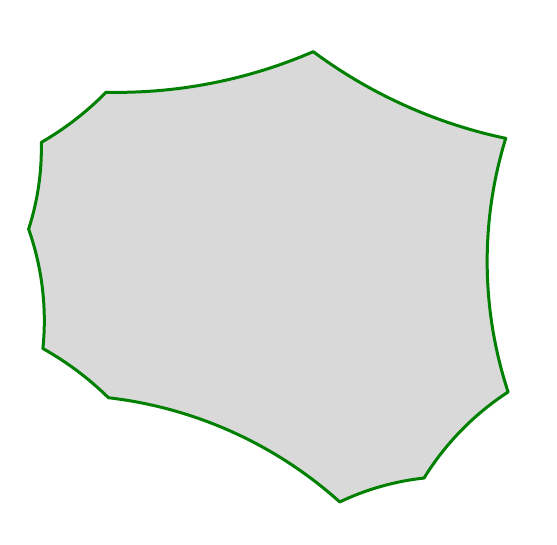}
	\caption{A fundamental domain for $\Gamma$.}
	\label{fig:fdom}
\end{figure}

We now introduce groupoids as
they allow us to express the spatial
component of the geometric presentations of 
$\pi_1(\Gamma\backslash(\h-Y))$ and are crucially used in the proofs of
Sections \ref{subsec:fgge} and \ref{sec:algos}.
Readers only interested in the application to cohomology 
computations may safely go to Section \ref{sec:cohomology}.

\subsection{Graphs}\label{subsec:graphs}
  A \emph{graph} $G$ is a compact
$1$ dimensional CW-complex. A $1$-cell of $G$ is called an
\emph{edge} and a $0$-cell is called a \emph{vertex}. We denote
by $E(G)$ (resp. $V(G)$) the set of edges of $G$ (resp. vertices)
of $G$. The gluing is represented by a map $\phi\colon 
E(G)\to V(G)\times V(G)$ where for each $e\in E$ the gluing
map $g\colon \{v,v'\}=\partial e\to V(G)$ is such that 
$\phi(e)=(g(v),g(v'))$. A map of graphs is a map of CW-complex.
A subset of edges $S$ of a graph $G$ naturally defines a 
subgraph of $G$ that we may denote again by $S$.

\subsection{Groupoids.}\label{subsec:groupoids}
Two basic references for this section are \cite{higgins, brown}. The language of category theory
is used freely and we take the ``fundamental group convention''
that arrows compose from left to right.

A groupoid $\G$ is a category where every morphism is an isomorphism.
A subgroupoid of a groupoid is a subcategory which is also a groupoid.
A morphism of groupoids is a functor of the underlying categories,
it is injective (resp. surjective, resp. bijective) if object and 
morphism maps are. We denote by $\Ob(\G)$ and $\Mor(\G)=\cup_{x,y\in \Ob(\G)} \G(x,y)$
the sets of objects and morphisms of $\G$ respectively, where $\G(x,y)$
denotes morphisms from $x$ to $y$.



\subsection{Fundamental groupoids.} \label{subsec:fgroupoid}
	The \emph{fundamental groupoid} $\pi_1(X, A)$ of a topological space $X$ with
	respect to a subset $A$ is the category
with objects $A$ and morphisms homotopy classes of paths between points of $A$.
The set $A$ is \emph{representative} in $X$ if it meets every connected components
of $X$. If $f\colon X\to Y$ is a continuous map we denote by $f_*$
the induced morphism $\pi_1(X, A)\to \pi_1(Y, f(A))$.
When $X$ is a graph we define the fundamental groupoid of $X$
to be $\pi_1(X, V(X))$.

\subsection{Groupoids to graphs and vice versa.}\label{subsec:ajdonction}
	From a groupoid $\G$ we can recover a graph $G=\Forget(\G)$ as follows:
    let $V(G)$ be $\Ob(G)$, let $E(G)$ be the set of morphisms
of $\G$ modulo inversion and without the 
identities, and let the gluing maps be given by source and target maps.
From a graph $G$, we can recover a groupoid by the fundamental 
groupoid functor.

The fundamental groupoid and forgetful functors are respectively left and right
adjoints. That is, $\pi_1(\_,V(\_))$ is the analogue of the 
free group functor for groupoids, where sets are replaced by
graphs as the underlying concrete category. For this reason
we also write $\Free(G)=\pi_1(G,V(G))$ for a graph $G$. If
$S$ is a subset of edges of $G$ we write $\Free(S)$ for the
smallest subgroupoid of $\Free(G)$ containing $S$.
\begin{rem}\label{rem:adj}
	This can be proven using the Van Kampen theorem
for groupoids that we introduce later and shows in particular that 
$\pi_1(\bigvee_{i=1}^f S^1, \bullet)=\Z^{*f}$ is the free group 
on $f$ generators.
\end{rem}
Each edge $e$ of a graph $G$ corresponds to exactly 
two paths in $\pi_1(G, V(G))$ that we call \emph{darts}. We
denote by $D(G)$ the set of darts of $G$. If $v$ 
is a vertex of $G$, then a dart $d$ is incident to $v$ if 
$d(1)=v$.

We now want to define presentations and kernels for groupoids.
\begin{rem}\label{rem:needuniversal}
Let $F\colon \G\to \Fs$ be a functor and $f$ and $g$ be two arrows
in $\G$.
It may happen that $F(f)F(g)$ is an identity in $\Fs$ but
$f$ and $g$ do not compose in $F$. To make sense of the 
expression $fg$ as a relation for $\Fs$, we need universal 
groupoids, which are introduced in the next section.
\end{rem}
\subsection{Universal groupoids.}\label{subsec:univgroupoids}
	Let $\G$ be a groupoid
	and let $h\colon \Ob(\G)\to X$ be a map of sets.
    Two morphisms $f,g\in \Mor(\G)$ \emph{eventually compose} in $X$ if 
    $h(s(g))=h(t(f))$ where $s$ and $t$ are source and target maps.
Define the \emph{universal groupoid $\mathcal U_h$ with respect to $h$} to be
the groupoid with object set $X$ and with morphisms the words of morphisms
in $\G$ that eventually compose in $X$.
If $f$ and $g$ already
compose in $\G$, then the words $(fg)$ and $(f,g)$ are
identified.
The map~$h$ induces a morphism of groupoids~$\G \to \mathcal U_h$.
For a functor $F\colon \G\to \F$, the universal groupoid associated
with $F$ is defined to be the universal groupoid associated with
$\Ob(F)$.
\begin{ex}
	Denote by $\textbf{I}$ the unit groupoid, that is the 
	fundamental
	groupoid $\pi_1([0,1], \{0,1\})$ of the unit real interval
	with respect to its boundary. The universal groupoid 
	with respect to the unique map $\{0,1\}\to \{1\}$ 
	is $\pi_1(S^1,\{1\})$.
\end{ex}

\subsection{Presentations for groupoids}\label{subsec:presgroupoids}
	Let $\G$ be a 
groupoid. A family of arrows $S$ of $\G$ 
\emph{generates} $\G$ if the smallest subgroupoid containing $S$ is $\G$. Equivalently, if
$S \to \Forget(\G)$ is a map of graphs such that 
$F\colon\Free(S)\to \G$ is surjective,
we say that $S$ \emph{generates} $\G$ and we write 
$\G = \langle S\rangle$. 
Define $\ker F$ to be the subgroupoid of $
\Free(S)$
with objects $\Ob(\ker F)=\Ob(\Free(S))$ and 
morphisms those 
morphisms in $f\in \Mor(\Free(S))$ such that $F(f)$ is an 
identity in $\G$. In this special case, any set of morphisms 
$R$ generating $\ker(F)$ 
is a set of relations for $\G$ and we write 
$\G=\langle S|R\rangle$.

\subsection{Quotient groupoid}\label{subsec:quotient}
The kernel $\mathcal N = \ker F$ of $F$ is a \emph{normal}
subgroupoid of 
$\Free(S)$, that is $\Ob(\ker F)=\Ob(\Free(S))$
and for every $x,y\in \Ob(\Free(S))$ and 
every $a\in \Free(S)(x,y)$ we have 
$a.\mathcal N(y,y).\bar a\subset \mathcal N(x,x)$. There is a notion of quotient
by a normal subgroupoid, and from
\cite[11.3]{brown}, $\G$ is isomorphic to the quotient
$\Free(S)/\mathcal N$.
For any surjective groupoid morphism 
$F\colon \G\to \G'$, the first isomorphism theorem holds, that is
$\G'\simeq \G/\ker(F)$. 

We can even assume that $F$ is not surjective but only that $F(\Mor(\G))$
generates $\G'$.
 Following Remark \ref{rem:needuniversal}, 
we should enhance the quotient $\G'\simeq \G/\ker(F)$ to a 
quotient $\G'\simeq \mathcal U/\ker(F_u)$ where $\mathcal U$ is 
the universal groupoid associated with $F$ and
$F_u\colon \mathcal U\to \G'$ is the canonical functor. Indeed
in this case $F_u$ is surjective.
Then, a presentation of $\G'=\langle S\mid R\rangle$ can be given
as a generating family of arrows for $\mathcal U$ and $\ker(F_u)$.

We now come to the main part of this section which will enable
us to compute explicit presentations of fundamental groupoids.

\subsection{The Van Kampen theorem.}\label{subsec:vankampen}
Let $i\colon X_0\hookrightarrow X_1$ and $f\colon X_0\to X_2$
be continuous maps with $i$ injective. Let $A$
be representative in $X_0$ and $B$ in $X_2$
with $f(A)\subset B$. Let $X$ denote a push-out of 
$X_1\xleftarrow{i} X_0\xrightarrow{f}X_2$. By a result of 
Brown \cite[Chapter 9.1]{brown}, when $i$ is a cofibration and in particular when $X_0$ is a sub-CW-complex of a CW-complex 
$X_1$, the induced diagram 
\[\begin{tikzcd}
	{\pi_1(X_0, A)} & {\pi_1(X_2, B)} \\
	{\pi_1(X_1, A)} & {\pi_1(X, B)}
	\arrow[from=1-1, to=1-2]
	\arrow[from=1-1, to=2-1]
	\arrow[from=1-2, to=2-2]
	\arrow[from=2-1, to=2-2]
\end{tikzcd}\]
is a push-out diagram.

\subsection{Pushouts.}\label{subsec:pushoutcomp}
Let $\mathcal B\xleftarrow{i}\mathcal A \xrightarrow{j} \mathcal C$ be a push-out
diagram of groupoids with $\G$ a push-out. Denote by 
$f\colon \mathcal B\to\G$ and $g\colon \mathcal C\to \G$ the push-out maps.
Let~$\mathcal U$ be the universal groupoid
associated with
$f\sqcup g\colon \mathcal B\sqcup \mathcal C\to \G$, and let
$p\sqcup q\colon \mathcal B\sqcup \mathcal C\to \mathcal U$ be the induced
morphism of groupoids.
Let $F\colon \mathcal U\to \G$ be the canonical functor.
Then $F$  is surjective and has kernel generated by the normalizer 
of the union of $\{(p\circ i)(a).(q\circ j)(\bar a))\mid a\in \Mor(\mathcal A)\}$, $\ker(f)$ and
$\ker(g)$ in $\mathcal U$. Let $S_{\mathcal A}$, 
$S_{\mathcal B}$, $S_{\mathcal C}$, $S_{\ker(f)}$ and $S_{\ker(g)}$
denote generators for $\mathcal A$, $\mathcal B$, $\mathcal C$,
$\ker(f)$ and $\ker(g)$ respectively. 
From Section \ref{subsec:quotient}
a set of generators for $\G$ is given by $S_{\mathcal B}\cup S_{\mathcal C}$
and a set of relations is given 
$S_{\ker}\cup S_{\ker(g)}\cup \{(p\circ i)(s)(q\circ j(\bar s))\mid s\in S_{\mathcal A}\}$. 
\section{Groupoids from graph embeddings}
\label{sec:topology}
In this section we revisit work of Imbert \cite{Imbert1,Imbert2} on the 
computation of geometric presentations for Fuchsian groups in
the language of groupoids and graph embeddings.

Throughout this section we fix a Fuchsian group $\Gamma$ and call $Y$ the set of
points of $\h$ with non trivial stabilizer under the action of $\Gamma$, also
called \emph{elliptic} points. We
assume that $\Gamma$ is co-compact for simplicity but the results
hold with small changes for finitely generated Fuchsian groups. We denote by 
$S$ the quotient space $\Gamma\backslash \h$.

Following Theorem \ref{thm:devissage}, we wish to compute the
automorphism group of the covering 
$\h-Y\to \Gamma\backslash (\h-Y)$. To do so we need to compute 
suitable presentations of $\pi_1(\Gamma\backslash(\h-Y))$ and $\pi_1(\h-Y)$. 
\subsection{Combinatorial graph embeddings.}\label{subsec:cge}
As mentioned in the paragraph above Remark \ref{rem:orblooporder}, the space
$S$ is an orientable and compact topological real surface. From a
theorem of Radó Tibor \cite{combtriangulation}, every such surface 
admits a combinatorial triangulation. Equivalently, every 
such surface admits a graph embedding 
$G\hookrightarrow S$ such that $S-G$ is a disjoint union of 
polygons with boundaries  removed and with the property that each edge
of $G$ lies in the boundary of at most two of these polygons.
Such an embedding is said to be \emph{combinatorial}.

 We fix a combinatorial graph embedding
$\iota\colon G\hookrightarrow S$ and an orientation on $S$. The
surface $S$ should now be 
thought of as a gluing of polygons along their edges.
Denote by $(P_i)_{i=1,\ldots,f}$ the polygons, or \emph{faces},
forming $S-G$.
Each $P_i$ is endowed with the orientation coming from $S$.

Let $E=\sqcup E(\partial P_i)$,
$ D=\sqcup D(\partial P_i)$ and $\sigma_1$ be a fixed-point free
involution of $E$ representing the gluing.
We can partition $E=E^+\sqcup E^-$ from
a choice of representatives in $E/\sigma_1$ 
as negatively and positively oriented edges 
where an edge
$e\in E(\partial P_i)$ is positively oriented if its orientation
comes from that of $P_i$. The orientation of $S$ also yields a 
cyclic ordering of 
$E(\partial P_i)$ for each~$i$. Put $\sigma_{2,i}$ to be the 
corresponding permutation of $E(\partial P_i)$ and $\sigma_2$
the induced permutation on $E$.

We respectively call $\sigma_1$ and $\sigma_2$ the \emph{side pairing} and \emph{face permutation} of $\iota$.

\begin{rem} The connectedness of $S$ is equivalent to the
	transitivity of the action of 
	$\langle\sigma_1,\sigma_2\rangle$ on $E$.
\end{rem}

The data of the permutations $(\sigma_1,\sigma_2)$ is a 
\emph{face representation} of $\iota$.
We will need the equivalent \emph{vertex representation}
$(\sigma_0,\sigma_1)$ with $\sigma_0=(\sigma_1\sigma_2)^{-1}$ 
being the \emph{vertex permutation} of the embedding.

The set of edges $E$ is in bijection with the set of darts
$D(G)$ of $G$. This way, each $\partial P_i$
can be represented by a loop or word $w_i=e_1\ldots e_l$,
where each $e_i$ is viewed as a dart in $D(G)$ that is positively
oriented in $P_i$, and 
$(e_1\ldots e_l)$ is the cycle of $\sigma_2$ corresponding to 
$P_i$. We write $E(w_i)$ for the set of edges $\{e_1,\dots,e_l\}$.

The set of edges of the words $\{w_i:i\}$ then has the property
that each edge of $G$ appears exactly twice, once
positively and once negatively. We then say the embedding 
$\iota$, and also $S$ is represented by the words 
$(w_i)_{i=1,\ldots, f}$. An $m$-handle word is a word of the form
$\prod_{i=1}^m [a_i, b_i]w'$ where the notation $\prod_{i=1}^m [a_i,b_i]$
means $a_1b_1\bar a_1\bar b_1\ldots a_m b_m\bar a_m\bar b_m$ and $w'$ is any
word.
\begin{ex}\label{ex:torusword}
	The $g$-holed torus has a face representation
	$(\sigma_1,\sigma_2)$ with side pairing $\sigma_1=\prod_{i=1}^g (a_i~\bar a_i)$
	and $\sigma_2=(a_1~b_1~\bar a_1~\bar b_1~\ldots~a_g~b_g~\bar a_g~\bar b_g)$. 
	In particular it is
	represented by the word 
	$w_g=a_1b_1\bar a_1\bar b_1 \ldots a_gb_g\bar a_g\bar b_g$.
\end{ex}

To understand the vertex representation, assume
$S$ has a $C^1$-structure and fix a vertex $v$ of $G$. 
Each edge of $G$ corresponds to two paths in $\pi_1(G, V(G))$.
Call a \emph{dart} of $G$ any path $d$ associated with an edge
and say that $d$ is \emph{incident to $v$} if $d(1)=v$.
Project a small
neighbourhood of $v$ in $S$ on the tangent plane of $S$ at $v$. This gives a
cyclic ordering of the darts incident to $v$.
Each cyclic ordering then corresponds to a cycle of $\sigma_0$.
The set of edges $E$ in face representations is now replaced
by the set of darts, that we call again $E$, and
$\sigma_1$ maps any dart
$d$ to $\bar d$, linking the vertices associated with $d$ and
$\bar d$.

\begin{rem}[Euler characteristic]\label{rem:eulerchar}
	If $\iota$ has $f$ faces, $v$ vertices, $e$ edge
	and $S$ has genus $g$ then the equality 
	$2-2g=v-e+f$ holds.
\end{rem}

\subsection{The dual graph embedding.}\label{subsec:dualgraph}
From the face representation $(\sigma_1,\sigma_2)$ we can obtain 
the face representation of another combinatorial graph embedding 
$\iota^*\colon G^*\hookrightarrow S$, called the \emph{dual} 
graph embedding of
$\iota$, with face representation
$(\sigma_1, \sigma_1\circ\sigma_0\circ\sigma_1)$. A vertex representation
of the dual graph embedding is given by $(\sigma_2,\sigma_1)$, so that it
lets us interprets faces (resp. vertices) of $\iota$
as vertices (resp. faces) of $\iota^*$.

The associated polygons and gluings are built as follows. Let 
$(P_i)_i$ be a set of polygons corresponding to 
$(\sigma_1,\sigma_2)$ and $E$ be its set of edges.
For each edge $e$, common
to $P$ and $P'$ through $\sigma_1$, link the centers of $P$ and
$P'$ by an edge
going through $e$. Denote the new edge by $e^*$.
Let~$G^*$ be the graph with set of edges 
$E^*=\{e^*|e\in E\}$ and set of vertices the centers of the $P_i$.
We get
a graph embeding $\iota^*\colon G^*\hookrightarrow S$
with vertex representation
$(\sigma_2,\sigma_1)$ if we naturally identify $E$ and $E^*$.
See Figure \ref{fig:dualgraph} which pictures dualizing an
embedding with a single face.

\begin{figure}
	\centering
	\includegraphics[width=1\textwidth]{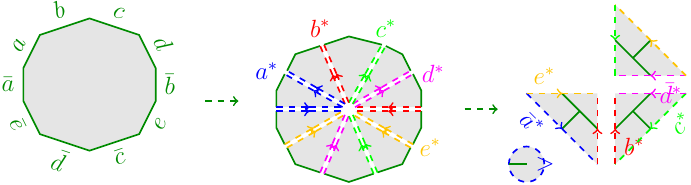}
	\caption{Dualizing $G\hookrightarrow S$ into $G^*\hookrightarrow S$.}
	\label{fig:dualgraph}
\end{figure}

\begin{rem}
	When $S=\Gamma\backslash \h$ and $\cup_i P_i$ form a fundamental domain for
	$\Gamma$, the set of polygons obtained by dualizing does not form a fundamental
	domain as elliptic points now lie in the interior of some polygons.
\end{rem}

\subsection{Our setting.}\label{subsec:setting}
	We are given a fundamental polygon $P$ with edges $E$ for a co-compact Fuchsian
	group $\Gamma$ with elliptic points $Y$ together with a side pairing 
	$(\sigma_1,\varphi)$. The polygon $P$ yields a combinatorial graph 
	embedding 
	$\iota\colon \partial P/\sigma_1=G\hookrightarrow S=\Gamma\backslash \h$ 
	with face representation $(\sigma_1,\sigma_2)$ on $E$. Let  
	$p\colon \h \to S$ be the quotient map.

	Recalling Theorem \ref{thm:devissage}, we need to puncture $S$ at points
	of $\Gamma\backslash Y$. But $Y$ meets $P$ at its vertices so that
	$\iota$ does not come from a map $G\hookrightarrow S-\Gamma\backslash Y$. Let 
	$\iota^*\colon G^*\hookrightarrow S$ be the dual graph embedding of
	$\iota$. Then $\Gamma\backslash Y$ does not meet $G^*$ so
	that $\iota^*$ comes from an embedding 
	$\iota^*\colon G^*\hookrightarrow S-\Gamma\backslash Y$.

	Remark that $\iota^*$ has a single vertex since
	$\iota$ has a single face. We have a groupoid
	surjection $\pi_1(G^*,V(G^*))\simeq \pi_1(S-V(G), V(G^*))\to \pi_1(S-\Gamma\backslash Y, V(G^*))$
	induced by the homotopy equivalence $G^*\hookrightarrow S-V(G)$ and 
	the inclusion $S-V(G)\subset S-\Gamma\backslash Y$.
	Following Theorem \ref{thm:devissage} and Remark
	\ref{rem:globalmonodromy}, we obtain
	a groupoid surjection 
	$\phi\colon \pi_1(G^*,V(G^*))\to \pi_1(S-\Gamma\backslash Y, V(G^*))\to \Gamma
	$ with kernel the image of the morphism 
	$\pi_1(\h-p^{-1}(V(G)), v_0)\to \pi_1(S-V(G), V(G^*))$
	induced by $p$ for some $v_0$ with $p(v_0)\in V(G^*)$.

	The group $\mathcal K = \pi_1(\h-p^{-1}(V(G)), v_0)$
	is obtained by first normalizing the groupoid $\G$ generated
	by the graph $\cup_ip^{-1}(\partial P_i)$ in $\pi_1(\h-p^{-1}(V(G)), p^{-1}(V(G^*)))$ and then letting $\mathcal K=\G(v_0)$.
	Explicitly, as in Section \ref{subsec:cge} the
	boundary
	$\partial P_i$ corresponds to a simple loop $\eta_i$ around
	some vertex
	$v$ of $P$ that maps through $p$ to a loop around 
	$p(v)$ a vertex of $G$. If 
	$c=(e_1 \ldots e_l)$ is the cycle of $\sigma_2$ corresponding
	to $p(v)$, then 
	$\delta_i=e_1\ldots e_l$ is a simple loop around $p(v)$. From 
	Remark \ref{rem:orblooporder} we then have 
	$p\circ\eta_i=\delta_i^m$ for some integer $m$ being the order
	of $\phi(\delta_i)$ in $\Gamma$.
	Then $\mathcal K$ is generated by the $\alpha\eta_i\bar \alpha$
	where $\alpha$ is a path
	from $v_0$ to $\delta_i(0)$. For each $i$ we denote 
	by $\gamma_i$ any of the loops $p_*(\alpha)\delta_ip_*(\bar\alpha)$.
	
	The map $\phi$ is fully determined by its image on $E^*$ 
	and we have $\phi(e^*)=\varphi(\bar e)$ for every $e^*\in E^*$. Indeed,
	let $g^*=\phi(e^*)$ and $g=\varphi(e)$. Also denote by $c$ the 
	center of $P$ and by $\theta^*$ the unique lift
	of $e^*$ with $\theta^*(1)=c$. Then as covering automorphisms
	of $p$ at $c$, both $g$ and $g^*$ are fully determined
	by their image at $c$. But $\theta^*.c$ is the
	center $g^{-1}.c$ of $g^{-1}.P$ by construction so
	that $g^*=g^{-1}$ and the result.

	We now show how to work and compute good presentations
	for the groupoid $\pi_1(S, V(G^*))$ before showing how to lift 
	them to the desired presentations for $\pi(G^*, V(G^*))$.

\subsection{Fundamental groupoids of surfaces}\label{subsec:fgS}
We keep the notations of the last section. We will compute
presentations of $\pi_1(S,V(G^*))$ coming from the map
$\iota^*$ but the procedure works for any combinatorial
embedding. A CW-complex structure on $S$ is given by 
$S^{(1)}=G^*$ and a $2$-cell for each $P_i$ with gluing maps 
induced by $\sigma_1$. Applying the Van Kampen theorem 
\ref{subsec:vankampen} to the push-out
\[\begin{tikzcd}
	{\bigsqcup_{i=1}^f\partial P_i} & {G} \\
	{\bigsqcup_{i=1}^f P_i} & S
	\arrow[from=1-1, to=1-2]
	\arrow[from=1-1, to=2-1]
	\arrow[from=1-2, to=2-2]
	\arrow[from=2-1, to=2-2]
\end{tikzcd}\]
we obtain the push-out
\[\begin{tikzcd}
	{\bigsqcup_{i=1}^f \pi_1(\partial P_i, V(\partial P_i))} & {\mathcal U_{\pi}/\mathcal N\langle e^*\sigma_1(e^*);~e^*\in E^*\rangle} \\
	{\bigsqcup_{i=1}^f (\pi_1(\partial P_i, V(\partial P_i))/\mathcal N\langle \delta_{i}\rangle)} & {\pi_1(S,V(G^*))}
	\arrow[from=1-1, to=1-2]
	\arrow[hook, from=1-1, to=2-1]
	\arrow[from=1-2, to=2-2]
	\arrow[from=2-1, to=2-2]
\end{tikzcd}\]
where $\mathcal U_{\pi}$ is the universal groupoid associated
with the projection $\pi\colon \bigsqcup_{i=1}^f \partial P_i\to\left( \bigsqcup_{i=1}^f \partial P_i\right)/\sigma_1$.

We obtain a presentation for $\pi_1(S,V(G^*))$ of the form
$\langle E^* \mid e\sigma_1(e),\delta_{i};~e^*\in E^*;~i\rangle$.
When $f=1$, there is a single face $P$ and 
writing $(E^*)^+=\{e_1<\ldots <e_n<e_1\}$ we get
$\pi_1(S,V(G^*))=\langle e_1,\ldots, e_n \mid e_1\ldots e_n\rangle$.

In the next section, we show how to reduce to this 
case.


\subsection{One face reduction.}\label{subsec:onefacered}
Let $G_{\textrm{faces}}$ be the graph with vertices $\{P_i;i\}$
and with edges the pairs
$(e^*,\sigma_1(e^*))$ such that $e^*$ and $\sigma_1(e^*)$ lie in
different polygons. Let $T$ be a spanning tree in 
$G_{\textrm{faces}}$. Define 
$P_T=\cup_{T} P_i$ to be the quotient of $\sqcup_{i=1}^f P_i$
where for every edge $(e,\sigma_1(e))$ of $T$ we identify $e$ and $\sigma_1(e)$.
See Figure \ref{fig:onefacered} for an explicit example.

The map $\partial P_T/\sigma_1\hookrightarrow S$ is again a 
combinatorial graph embedding. If we write $E(\partial P_T)^+=\{e_1<\ldots<e_{n'}<e_1\}$ and let $\delta=e_1\ldots e_{n'}$,
we have proved that 
\begin{align*}
\pi_1(S,V(G^*))&=\langle E^*| e^*\sigma_1(e^*),~ \delta_{i};~e^*\in E^*,~i=1,\ldots,f\rangle\\
			 &=\langle e_1,\ldots, e_{n'}| \delta \rangle
\end{align*}
In the following we let $\gamma=\alpha\delta\bar \alpha$ for
some path $\alpha$ from $v_0\in V(G^*)$ to $\delta(0)$.

\begin{rem}
If we denote by $n_i$ the number of edges of $P_i$ for each $i$
then $P_T$ has $n'=\sum_i n_i - 2(f-1)$ edges. 
When $G$ has a single
vertex, $P_T$ has a single vertex and computing the
euler characteristic (Remark \ref{rem:eulerchar}) of $\iota_T$ gives $n'=4g$ where $g$ is the genus of $S$. 
\end{rem}

\begin{figure}
	\centering
	\includegraphics[width=1\textwidth]{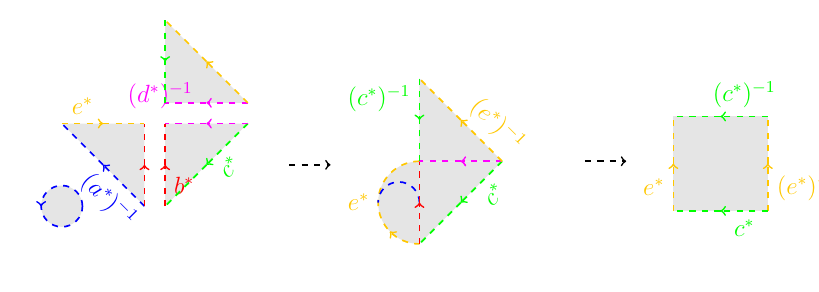}
	\caption{One face reduction into $cec^{-1}e^{-1}$.}
	\label{fig:onefacered}
\end{figure}
\subsection{Fundamental groupoids of graph embeddings.}\label{subsec:fgge}
	The groupoid surjection $\pi_1(G^*, V(G^*))\to \pi_1(S, V(G^*))$ induced by $\iota^*$ has kernel 
$\mathcal K=\mathcal N\langle \gamma_i:i=1,\ldots,f\rangle$ from
the last section.  
Replacing $\iota^*$ by 
$\iota_T\colon G_T=\Gamma\backslash \partial P_T\hookrightarrow S$ 
in the last section we can show that the morphism
$\pi_1(G_T, V(G_T))\to \pi_1(S,V(G^*))$ has kernel 
$\mathcal N\langle \gamma\rangle$. Let $F_T$ be the image of $\sqcup \partial P_i$
in $P_T$. Then, $\iota_T$ factors
through $\iota^*$
in such a way that it lifts to an embedding 
$\partial P_T\hookrightarrow F_T$
with image $F_T-T$ if $T$ is naturally seen as a subgraph of $F_T$.
In particular $\gamma$ lies in $\mathcal K$ so that it is of the
form $\gamma=\prod g_j \delta_j^{\pm 1} \bar g_j$ for some paths
$g_j$ in $\pi_1(G^*,V(G^*))$ from $\gamma(0)$ to $\delta_j(0)$. 

In Section \ref{sec:algos} we will show that in fact 
we can obtain the following.
\begin{thm}[Fundamental groupoids of graph embeddings]\label{thm:fgge}
	Any combinatorial graph embedding $\iota\colon G\hookrightarrow S$
	represented by words $(w_i)_i$ gives a presentation
	of the fundamental groupoid of $G$ of the form
	\[\pi_1(G ,V(G))=\langle \sqcup_i E(w_i)\mid w_i\delta_i, i\rangle\]
	with $\delta_i=\overline{w_i}$ and $\ker(\pi_1(G,V(G))\to \pi_1(S, V(G))=\mathcal N\langle \delta_i;i\rangle$. 
	Further, let $G_{\textrm{faces}}$ be the graph with vertices those faces
	of $\iota$ and vertices those pairs of edges $(e,\sigma_1(e))$ such that $e$
	and $\sigma_1(e)$ lie in different faces. Then any spanning tree $T$
	in $G_{\textrm{faces}}$ yields another presentation
	\[\pi_1(G,V(G))=\langle E(w), g_1\delta_{h(1)}\bar g_1,\ldots, g_f\delta_{h(f)}\bar g_f\mid w. \prod_{j=1}^f g_j \delta_{h(j)} \bar g_j=1\rangle\]
	for a word $w$ representing $S$ and we can compute explicitly the ordering $h$ 
	of the product, the paths $g_j$ and the word $w$ as follows.

	For any $i=1,\ldots, f$, let $P_i$ be a polygon corresponding
	to $w_i$. Let $P_T$ be the quotient of $\sqcup_i P_i$ 
	obtained by gluing $e$ and $\sigma_1(e)$ for any edge
	$(e,\sigma_1(e))$ in $T$ and $F_T$ be the image of $\sqcup_i \partial P_i$ in
	$P_T$. If $v=w(1)$, then:
	\begin{itemize}
		\item $h\colon \{1,\ldots, f\}\to \{1,\ldots, f\}$
	is the depth-first search ordering in $T$ from the smallest face containing
	$v$ using the cyclic ordering on vertices of $T$ induced by $\sigma_2$.
		\item For each $j=1,\ldots, f$, the path
			$g_j$  is the shortest path in $F_T$ from $v$ to $\delta_{h(j)}(0)$ respecting 
	orientation.
		\item The word $w$ represents the graph embedding $\Gamma\backslash\partial P_T\hookrightarrow S$.
	\end{itemize}
	When $w$ is an
	$m$-handle word for $S$, we call such presentation an
	$m$-handle presentation of $\pi_1(G,V(G))$. The special
	cases $m=0$ and $m=g$ for $g$ the genus of $S$ are
	respectively called one word and geometric presentations 
	of $\pi_1(G,V(G))$.
\end{thm}
From Theorem \ref{thm:devissage} and Remark 
\ref{rem:orblooporder} the next well-known theorem follows.
\begin{cor}\label{cor:presoneword}
	Let $\Gamma$ be a co-compact Fuchsian group with signature
	$(g;\nu_1,\ldots, \nu_f)$. Then there exists a word 
	$w=e_1\ldots e_n$
	representing $\Gamma\backslash \h$ and a presentation
	\[\Gamma\simeq \langle e_1,\ldots, e_n, \gamma_1,\ldots,\gamma_f\mid w.\prod _{i=1}^f\gamma_i, \gamma_j^{\nu_j};j\rangle\]
\end{cor}
We will say that the generators $(e_i)_{i=1,\ldots, n}$ are
topological while the $(\gamma_i)_{i=1,\ldots, f}$ are
elliptic.

The rest of this section is dedicated to building $m$-handle
words representing $S$ from any word representing $S$. We 
anticipate the last result in the next corollary.
\begin{cor}\label{cor:pres}
	In the setting of Corollary \ref{cor:presoneword}, $w$ can 
	be replaced by an $m$-handle word representing $\Gamma\backslash \h$ for any $m=0,\ldots, g$.
\end{cor}
\subsection{Cut and paste.}\label{subsec:cutandpaste}
Let $w$ be a word representing a surface $S$ and 
$C\hookrightarrow S$ the associated graph embedding. Let $a,b$ be
two edges in $w$ such that we can write $w=a\alpha b\beta \bar a\gamma \bar b\delta$.
By cutting along $a$ and pasting along $b$ we mean
representing $S$ by the new word 
$w'=ad\bar a\gamma\beta \bar d\alpha\delta$ where
$d=\alpha b \beta$ in the underlying groupoid of $S$. In 
practice, this operation can be viewed
 as cutting the polygon attached to
$w$ following an arc joining $a(1)$ and $\bar a(0)$ and 
gluing the $2$ obtained polygons along $(b,\bar b)$, 
see Figure \ref{fig:cutandpaste}, then we always have 
$w=w'$ as paths in the fundamental groupoid of $S$. Denote by 
$C'\hookrightarrow S$ 
the graph embedding attached to $w'$. This 
operation translates into a groupoid isomorphism 
$F_{a,b}\colon \pi_1(S,V(C'))\to \pi_1(S,V(C))$ defined by 
$F_{a,b}(c)=\bar \beta \bar \gamma a$, $F_{a,b}(d)=\alpha b\beta$.
and $F_{a,b}(e)=e$ for $e$ not in $\{c,d,\bar c, \bar d\}$.
\begin{figure}
	\centering
	\includegraphics[width=0.6\textwidth]{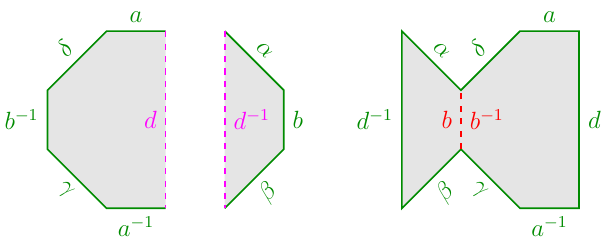}
	\caption{Cutting $a\alpha b\beta \bar a\gamma \bar b\delta$ along $a$ and pasting along $b$.}
	\label{fig:cutandpaste}
\end{figure}
\subsection{Extracting a handle.}\label{subsec:onehandle}
    Let $w$ be a word representing $S$ 
    such that there exists edges $a$ and $b$ with
    $w=a\alpha b\beta\bar a\gamma \bar b\delta$.
    First cutting along $c=\alpha b\beta$ and
    pasting along $b$ yields the word 
    $w'=ac\bar a\gamma\beta \bar c\alpha\delta$,
    then cutting along $\bar d=\bar a\gamma\beta$ and
    pasting along $\bar a$ yields the word
    $w''=\gamma\beta d c \bar d\bar c\alpha\delta$
	where the pattern $[d,c]=dc\bar d\bar c$ is called a handle.

When $w$ has a single vertex, such $a$ and $b$ always exist. 
We now suppose $w$ has a single vertex.
Noting that any handle already present in 
$\alpha\beta\gamma\delta$ remains a handle after extraction,
by successive extraction of handles we obtain a canonical 
word $w_g=\prod_{i=1}^g [a_i, b_i]$. But this process
is not suitable for algorithmic purposes. The following
algorithmic procedure was proposed by M. Imbert in \cite{Imbert1}.
\subsection{Extracting multiple handles.}\label{subsec:mhandles}
	We show how to build an $m$-handle word for $S$ for
	any $m=0,\ldots, g$ where $g$ is the genus of $S$.
	Assume recursively that we have built a word
	$w=\zeta c.\alpha.d.\beta \bar c\gamma \bar d.\beta$ with
	$\zeta=\prod_{i=1}^{m-1} [a_i,b_i]$ for $1\leq m$. 
	Cutting along $c$ and pasting along $d$ gives
	$\zeta c.e.\bar c\gamma\beta \bar e\alpha\delta$,
	cutting along $\bar c$ and pasting along $\bar e$ gives
	$c\alpha\delta\zeta f\gamma\beta \bar c\bar f$. Now
	cutting along $f$ and pasting along $\bar c$ gives
	$\zeta f\bar g\bar fg\gamma\beta\alpha\delta$ and letting
	$h=\bar g$ concludes the recursion. 

	Gathering the cut 
	and pastes we can explicitly
	write 
	$f=\overline{\gamma\beta\overline{(\alpha.d.\beta)}\alpha\delta\zeta}$
	and  $h=\overline{\gamma\beta \bar c}$.

	\begin{rem}
	The case $m=g$ in Section \ref{subsec:mhandles} yields the
	well-known geometric presentation.
	\end{rem}

\section{Algorithms}\label{sec:algos}
We keep the notations of Section \ref{subsec:setting}.
\subsection{Representation of the data.}\label{subsec:datarepr}
	We represent graph embeddings by their face representations, that is two
	permutations on a set of edges. From the side pairing of $\Gamma$ we can
	represent any element of $\Gamma$ by a word of edges in $E(G)$. 
	Instead of words of edges we use \emph{straight line programs} 
    (see Section~\ref{subsec:slp} and \cite[Section 3.1.3]{handbookgroup}) to
    represent elements and families of elements of  $\Gamma$.
	This will enable us to represent the set of conjugating paths $\{g_j\}_j$ in
	quasi-linear space instead of quadratic and to represent canonical
	loops in quasi-quadratic space instead of exponential space
	for the geometric presentation.
\subsection{Straight line programs.}\label{subsec:slp}
    Fix $1\leq n$ an integer.
    A \emph{straight line program} $s$ is a vector such that
	each entry is of the form $(\op,i,j)$, $(\op,i)$ or $(1)$
	for some symbol op representing a unary or binary operation
	and some integers $i$ and $j$. 

	Fix a straight line program $s$, a group $\Gamma$
	and a finite family $(g_i)_{i=1, \ldots, n}$
	of elements of $\Gamma$. The evaluation of $s$ in $\Gamma$ is
	done as follows: let $m$ be the length of $s$ and
	let $o$ be a vector of size $n+m$ such that $o_i=g_i$ for $i=1,\ldots,n$
	and $o_i=1_\Gamma$ otherwise. We will use the symbols mul, inv and
	id to represent multiplication, inversion and identity.

	We will incrementaly fill $o$ with
	elements of $\Gamma$. 
	Assume $o$ has been built up to $o_{n+k-1}$ for $k\geq 0$, let 
	$\textrm{instr}=s_k$, then:
	\begin{enumerate}
		\item If $\textrm{instr}=(\textrm{mul},i,j)$, put $o_k=o_io_j$.
		\item If $\textrm{instr}=(\textrm{inv},i)$, put $o_k=o_i^{-1}$.
		\item If $\textrm{instr}=(\textrm{id},i)$, put $o_k=o_i$.
		\item $\textrm{instr}=(1)$, put $o_k=1_\Gamma$.
	\end{enumerate}
The output of such algorithm can then either be the whole vector
$o$ or if a vector of pointers $\rho$ pointing to a set of
specific elements is given, the vector $(o_{\rho_j})_{j}$.
\begin{ex}
	Let $\Gamma=\Z$ and $(g_1,g_2)=(13,11)$. Let $+$ denote $\textrm{mul}$ and $-$ denote $\textrm{inv}$ for 
	$\Z$. The straight 
	line program
    \[
      s=[(+, 1,1), (+,3,3), (+, 4, 1), (-, 5), (+, 2,2), (+, 7, 2), (+, 8, 8),
      (+, 6, 9)]
    \]
    evaluates to $o=[26, 52, 65, -65,22, 33, 66, 1]$ in $\Z$ and corresponds
	to the identity $1=-5\cdot 13+6\cdot 11$.
\end{ex}
\subsection{Summary of the setup.}\label{subsec:output}
Recall that we are given a co-compact Fuchsian group $\Gamma$ as a fundamental
polygon $P$ and a side pairing $(\sigma_1,\varphi)$ with $S=\Gamma\backslash \h$.
From this fundamental polygon
we obtain a combinatorial graph embedding 
$\iota\colon G=\Gamma\backslash \partial P\hookrightarrow S$
and dual graph embedding $\iota^*\colon G^*\hookrightarrow S$ such that the induced
morphism of groups $\pi_1(G^*,V(G^*))\to \pi_1(S-V(G),V(G^*))$ is an isomorphism.

To compute an $m$-handle presentation (\ref{thm:specialpres})
for 
$\Gamma$, from Theorem \ref{thm:devissage}, we can view $\Gamma$ as the
automorphism group $\Gamma(p)$ of the covering 
$p\colon \h-Y\to S-\Gamma\backslash Y$ for 
which we have 
$\Gamma(p)\simeq \pi_1(S-\Gamma\backslash Y, p(v_0))/p_*\pi_1(\h-Y, v_0)$.
From Section \ref{subsec:setting} we further have $\pi_1(S-\Gamma\backslash Y)/p_*\pi_1(\h-Y, v_0)\simeq \pi_1(S-V(G), V(G^*))/p_*\pi_1(\h-p^{-1}V(G), V(G^*))$ 
where a generating family for $p_*\pi_1(\h-p^{-1}V(G), V(G^*))$ is
described again in Section \ref{subsec:setting}.
We then only need to compute a presentation for $\pi_1(G^*,V(G^*))$
as in Theorem \ref{thm:fgge}.

We use the notations of Section \ref{subsec:onefacered} and Section 
\ref{subsec:fgge}. The gluing map $\sigma_1$ induces a projection 
map $\pi\colon F_T\to G^*$ such that the induced morphism
$\pi_1(F_T, V(F_T))\to \pi_1(G^*, V(G^*))$ has kernel
 of the form 
$\langle  e\sigma_1(e);e\in E(F_T)-E(T)\rangle$ so
that we may only work in $\pi_1(F_T,V(F_T))$. We fix $v_0\in V(F_T)$.

\subsection{Computation of $\pi_1(F_T,v_0)$.}\label{subsec:algos}
Before computing $\pi_1(F_T,v_0)$, we set some notations and build
an ordering $h$ of the faces $(P_i)_{i=1,\ldots, f}$ as in Theorem 
\ref{thm:fgge}.

Assume that a spanning tree $T$ in 
$G_{\textrm{faces}}$ is given and denote by $E$ the set 
$\sqcup_i E(\partial P_i)$. Let $A$ be the set of
	vertices of the form $e(0)$ for $e$ a positively oriented
	dart of some polygon $\partial P_i$ and such that 
	$(e,\sigma_1(e))$ is an edge of $T$. We also let $v_0\in A$.
	For each vertex 
	$v\in A$ there is a unique corresponding dart $e$
	so that we may write $e\in A$. Viewing $T$ as a tree in 
	$\iota\colon G\hookrightarrow S$ linking the centers of the polygons 
	$(P_i)_{i=1,\ldots, f}$, the permutation 
	$\sigma_2^*$ yields a cyclic ordering on the darts incident
	to a vertex of $T$.

	We will order the set $A$ as well as the polygons $(P_i)_i$
	as in Theorem \ref{thm:fgge} using a depth first search in $T$.
 	Let $e_0$ be the dart corresponding to $v_0\in A$.
	\begin{algo}[Ordering $A$]\label{algo:dfsordering}
		$~$
	\begin{itemize}
		\item Input: $A$, $e_0$, $(\sigma_1,\sigma_2)$, an index
			$i$ initialized at $1$ for recursion
			and an ordering of the faces
			$(P_j)_j$.
		\item Output: an ordering of $A$.
	\end{itemize}
	\begin{enumerate}
		\item Initialize $e\gets e_0$.
		\item Mark the face of index $i$.
		\item If $e$ is in $A$ and the face $P_k$ containing
			$\sigma_1(e)$ is not marked, recurse with $e_0$
			replaced by $\sigma_1(e)$ and $i$ initialized at $k$.
			Then go to step (4).
		\item Increment $e$ to $(\sigma_2^*)^{-1}(e)$. If
			$e=e_0$, terminate. Otherwise go to step (3). 
	\end{enumerate}
	This yields an ordering $h\colon \{1,\ldots, f\}\to A$. From the recursion
	process, the set $\sigma_1(A)=\{\sigma_1(e)|e\in A\}$ is such that
	for each face $P_k$ there is a unique corresponding edge 
	$\sigma_1(e)\in\sigma_1(A)$. The ordering $h$ then also gives an ordering 
	of the faces $(P_{h(i)})_{i=1,\ldots, f}$. We also let 
	$\phi\colon \{1,\ldots, f\}\to\{1,\ldots, f\}\cup \{0\}$ associating
	with  a face $P_i$ with $i\geq 2$ the index of its immediate predecessor in the depth first search, or $0$ if $i=1$.
	\end{algo}

	 For $1<j$, fix $e$ to be the dart of index $j$ in $A$ and $P_j$ be the 
	corresponding polygon. There is a unique positively
	oriented path in $\sqcup_j \partial P_j$, which may be constant, from
	$\sigma_1(\phi(e))(1)$ to $e(0)$. In $F_T$ this is a 
	path from $\phi(e)(0)$ to $e(0)$ that we call $\alpha_j$. 
	As in Section \ref{subsec:setting}, let $\delta_j$ be the 
	only negatively oriented
	simple loop in $\partial P_j\subset F_T$ starting at $e(0)$.

Let $H$ be the subgraph of $\pi_1(F_T, A)$ with edges 
$\{\alpha_j;j\}\cup \{\delta_j;j\}$. Then $V(H)=A$. See 
Figure \ref{fig:gis} for an example.

The graph $H$ contains a unique 
spanning tree $T'$ which is isomorphic to $T$, such that
$H-T'$ is the subgraph of loops $\{\delta_j\}_j$. Finally
for $1<j$, let $g_j$ be the unique
smallest path from $v_0$ to $\delta_j(0)$ in $H$ and 
$\gamma_j=g_j\delta_j\bar g_j$.
For $j=1$, put $g_1$ to be the constant path and 
$\gamma_1=\delta_1$. We are ready to prove the main lemma from
which Theorem \ref{thm:fgge} follows.
\begin{lem}\label{lem:mainlemma}
	Let $\gamma$ be the only simple path in $\partial P_T=F_T-T$
	in negative orientation starting at $v_0$. As a word of
	edges, $\gamma$ represents $S=\Gamma\backslash \h$ and
	the equality $\gamma=\prod_j \gamma_j$ holds in 
	$\pi_1(F_T, v_0)$ and a fortiori in $\pi_1(G^*, V(G^*))$.
\end{lem}
\begin{proof}
	The graph $H$ is a generating graph for $\pi_1(F_T,A)$ so 
 	that we only need to
	find the relation 
	$\gamma=\prod_{j=1}^f g_j\delta_{j}\bar g_j$ in $\pi_1(H, A)$. 
 	Let $\Pi=\pi_1(H, A)$. The group $\Pi(v_0)$ is freely generated
	by the set $\{\gamma_j\}$ as $H/T'\simeq \vee {i=1}^f S^1$ and $T'$
	is a spanning tree.

	We say that a word
	$u_1\ldots u_k$ of darts in a graph $G$ is elementary if there is no $j$ such 
	that $u_j=\bar u_{j+1}$ in $\pi_1(G,V(G))$.

	We will show that 
	$\gamma_f\bar \gamma=\bar \gamma_{f-1}\ldots \bar\gamma_{1}$.

	Let $(j_i)_{i=0,\ldots, l}=(\phi^i(f))_{i=0,\ldots, l}$ denote the path from 
	$P_1$ to $P_{h(f)}$ in $G_{\textrm{faces}}$. Then for any $i=0,\ldots,l$,
	the path $\alpha_{j_i}$ does not cross $T$ since by 
	construction each
	$j_i$ corresponds to some edge $e$ in a face $P$ which is seen last at
	this depth of the recursion in the depth first search of Algorithm
	\ref{algo:dfsordering}.

	It follows that $g_f=\alpha_{j_l}\ldots \alpha_{j_0}$ does not
	cross $T$ so that if we write $\delta_f=\bar e\eta_f$ for $e=h(f)$
	then $g_f\bar \eta_f$ is elementary and contained in $\partial P_T$. We can
	then write $\bar \gamma= g_f\bar \eta_f \eta$ and the product 
	is elementary. Then $\gamma^{(1)}=\gamma_f\bar \gamma=g_f\bar e\eta$ is a simple loop 
	in $F_T$ with image $F_T-\left((T-\{e\})\cup E(\eta_f)\right)$ and with $\gamma^{(1)}(0)=v_0$.

	We can then proceed by recursion by writing $A_1=A-\{h(f)\}$, $T_1=T-{e}$,
	$P_{T_1}$ the gluing of $(P_i)_{i=1,\ldots, f-1}$
	along $T_1$, $F_{T_1}$ the image of the graph $\cup_{i=1}^{f-1} \partial P_i$
	in $P_{T_1}$ and $H_1=H-\{\alpha_f, \delta_f\}$. Indeed since
	$h|_{\{1,\ldots,f-1\}}$ is the depth first
	search ordering of $A_1$ associated with $T_1$ starting at $v_0$, 
	the analogous statement in $H_1$ says that 
	$\gamma^{(1)}=\gamma_1\ldots\gamma_{f-1}$ which is the desired result. 
\end{proof}

\begin{figure}
	\centering
	\includegraphics[width=0.9\textwidth]{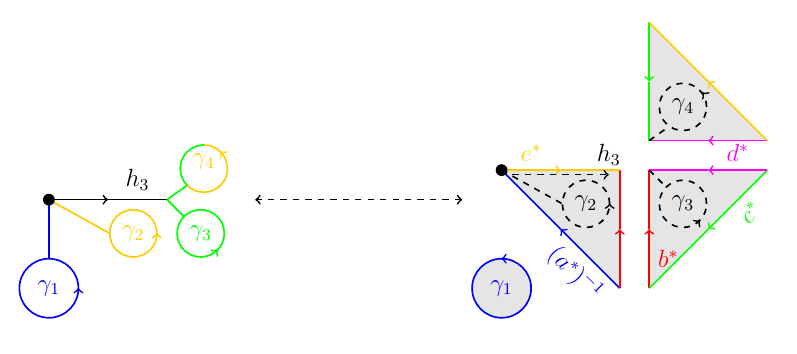}
	\caption{The graphs $H$ and $F_T$ with $\gamma=\gamma_1\gamma_2h_3\gamma_3\gamma_4\overline{h_3}$.}
	\label{fig:gis}
\end{figure}
It remains to show how to build $T$ and straight line 
programs for the paths $(g_j)_{j=1,\ldots, f}$ and the word
$\gamma$. Let $e_0$ be the dart corresponding to $v_0\in P_1$.
\begin{algo}[Building T]\label{algo:T}
The algorithm is almost identical to Algorithm 
\ref{algo:dfsordering} and can be performed at the same time.
\begin{itemize}
	\item Input: the dart $e_0$ and $(\sigma_1,\sigma_2^*)$.
	\item Output: a tree $T$ in $G_{\textrm{faces}}$.
\end{itemize}
	\begin{enumerate}
		\item Mark the face $\partial P_1$. Let $e=e_0$.
		\item If $\sigma_1(e)$ is not marked, add $(e,\sigma_1(e))$
			to $T$ and recurse with $e_0$
			replaced by $\sigma_1(e)$ then go to step 2.
		\item Increment $e$ to $(\sigma_2^*)^{-1}(e)$. If
			$e=e_0$, terminate. Otherwise go to step 1. 
	\end{enumerate}
\end{algo}

\noindent We write $E(\partial P_T)=\{e_0,\ldots, e_{n'}\}$ where the ordering
comes from the positive orientation of $P_T$. We also write
$E(T)=\{(e^2, \ldots, e^f)\}$ with the ordering of the 
depth first search and if $\varphi$ is the morphism
$\pi_1(G^*,V(G^*))\to \Gamma$ we write
$\varphi(D(G^*)^+)=\{\beta_1,\ldots, \beta_m\}$ which is a
generating family associated with the embedding 
$G\hookrightarrow S$.
\begin{algo}[Building a straight line program computing $(g_i)_{i=1,\ldots, f}$.]\label{algo:gis}
	$~$
	\begin{itemize}
		\item Input: $T,~F_T$ and $(\sigma_1,\sigma_2^*)$.
		\item Output: a straight line program $s$ 
			with pointers $\rho$ such that $s_{\rho_i-n}$ computes $g_i$.
	\end{itemize}
	\begin{enumerate}
		\item Let $s$ be a vector of size $2n$. Let $\rho$ be 
	a vector of pointers of size $f$. As $g_1=\emptyset$ let 
	$s_1=(1)$ and $\rho_1=1$.
	\item Let $k=2$.
	\item For $l=2$ to $f$:
		\begin{itemize}
			\item Let $e=\sigma_1(e^l)$
and $i=\rho(\phi(e^l))$, that is $i$ points to the path $g_i$
joining $v_0$ to the face of index $l-1$ in $o$.
			\item If $e(1)=e^{l+1}(0)$ put $s_k=(\id_o, n+i)$,
				$\rho(l)=k$, $k=k+1$.
			\item Otherwise, until $e=e^{l+1}$:
				\begin{itemize}
					\item If $\varphi(e)=\beta_j$,
				put $s_k=(\textrm{mul}_o, n+i, j)$,
				$\rho(l)=n+k$, $e=\sigma_2^*(e)$ and 
				$k=k+1$.
				\end{itemize}
			\item Put $\rho(l)=k-1$.
		\end{itemize}
		\item Return $(s, \rho)$.
	\end{enumerate}
\end{algo}
\noindent We now fix $E(F_T)=\{d_1,\ldots, d_{n'}\}$ an arbitrary ordering.
\begin{algo}[Building the word $\gamma$ from T]\label{algo:gamma}
	$~$
	\begin{itemize}
		\item Input: $T,~F_T$ and $(\sigma_1,\sigma_2^*)$.
		\item Output: the word $\gamma$ representing $\Gamma\backslash \partial P_T\hookrightarrow S$.
	\end{itemize}
If $S$ has genus $0$ return the empty word.
	\begin{enumerate}
		\item Let $e=e^1$.
		\item While $e$ is in $T$:
			\begin{itemize}
				\item Let $e=(\sigma_2^*)^{-1}\sigma_1(e)$.
			\end{itemize}
		\item Then $e=e_0$ and we put $\gamma=e$.
		\item For $i=1,\ldots, n$:
			\begin{itemize}
				\item Let $e=(\sigma_2^*)^{-1}(d_i)$.
				\item While $e$ is in $T$:
					\begin{itemize}
						\item Let $e=(\sigma_2^*)^{-1}(\sigma_1(e))$.
					\end{itemize}
				\item Let $\gamma=\gamma.e$.
			\end{itemize}
		\item Return $\gamma$.
	\end{enumerate}
\end{algo}
Once the one face reduction has been performed and we have 
obtained a word $w$ representing $S$, obtaining the various
special presentations is simply
a matter of rewriting the topological generators of the 
fundamental group of $\pi_1(G^*,V(G^*))$ as explained in 
Section \ref{subsec:cutandpaste}. The word $w$ is then 
rewritten as a word $w'$ in the new generators such that 
$w=w'$ in $\pi_1(S-V(G),V(G^*))$, and hence in 
$\pi_1(G^*,V(G^*))$.

In the setup of straight line programs, it is more convenient to have a single
straight line program $s$ on a set of generators for $\Gamma$ corresponding to
the side pairing associated with $P$ that computes the family 
$(e_1,\ldots, e_{n'}, \gamma_1,\ldots, \gamma_f)$ forming an $m$-handle presentation
of $\Gamma$. Recall that the evaluation of $s$ depends on the chosen group
so that we may say $s$ computes $(e_1,\ldots, e_{n'}, \gamma_1,\ldots, \gamma_f)$
from $E(F_T)$ to say the same thing as $s$ computes 
$(\varphi(e_1),\ldots, \varphi(e_{n'}), \varphi(\gamma_1),\ldots, \varphi(\gamma_f))$ from $\varphi(E(F_T))$.
\begin{algo}[Building the full straight line program]\label{algo:mainslp}
	$~$
	\begin{itemize}
		\item Input: $(s,\rho)$ as output by Algorithm \ref{algo:gis}, $\gamma$
			and $(e_1,\ldots, e_{n'})$ output by Algorithm \ref{algo:gamma}.
		\item Output: straight line program from $E(F_T)$
			to $(e_1,\ldots, e_{n'}, \gamma_1,\ldots, \gamma_f)$.
	\end{itemize}
	\begin{enumerate}
		\item Let $\rho'$ be pointers to $\{e_1,\ldots, e_{n'}\}\subset E(F_T)$
			using an ordering of $D(F_T)$ and $s'=((\textrm{id}, \rho'(i)))_{i=1,\ldots, n'}$. 
		\item Let $s'$ be the composition of $s'$ with $s$ to get a straight line 
			program from $D(F_T)$ to $(g_j)_j$ and update $\rho'$.
		\item Let $s'$ be the composition of $s'$ and a straight line
			program from $(\delta_j,g_j)_j$ to $(\gamma_j)$ to get a straight line
			program from $D(F_T)$ to $(\gamma_j)_j$ and update $\rho'$.
		\item Concatenate $\rho'$ with pointers to the
			instructions computing 
			$(e_1,\ldots, e_{n'})$ in $s'$ 
			to get an slp from $D(F_T)$ to 
			$(e_1,\ldots, e_{n'}, \gamma_1,\ldots, \gamma_f)$.
		\item  Return $(s',\rho')$.
	\end{enumerate}
\end{algo}

\begin{rem}\label{rem:invslp}
	For various purposes it may be useful to invert a given straight line
	program $s$. If $s$ computes a family $(\theta_j)_j$ from a family 
	$(\beta_j)_j$ of elements of $\Gamma$. By inverting $s$ we mean building a
	straight line program $s^{-1}$ computing $(\theta_j)_j$ from $(\beta_j)_j$.
	In general it is very hard but in our case it is straight forward. 
	in our case. Let $(s,\rho)$ be output by Algorithm \ref{algo:mainslp}
	computing $(e_1,\ldots, e_{n'},\gamma_1,\ldots, \gamma_f)$ from $E(F_T)$.
	As $\{e_1,\ldots, e_{n'}\}=E(F_T-T)$ we only need to show how to write
	edges of $T$ as seen in $F_T$. By definition
	$\gamma_f$ corresponds to a face $P$ which is a leaf of $T$. That is 
	a single edge $e$ of $P$ is in $T$. Further, as shown in the proof
	of Lemma \ref{lem:mainlemma}, $\bar g_f$ is made only of
	edges in $E(F_T-T)$. But then since
	$\delta_f= u_1\ldots u_k e$ corresponds to $\partial P$ with 
	$u_1,\ldots, u_k\in E(F_T-T)$ we can write 
	$e=\bar g_f\gamma_fg_f\bar u_k\ldots\bar u_1$ and proceed by recursion.
\end{rem}

\subsection{Complexity analysis. }\label{subsec:complexity1}
	Let $n$ be the number of edges of the fundamental polygon $P$. Building 
	$T$ has complexity that of the depth first search which
	is $O(n+f)=O(n)$.
	When faces have size $\leq m$, building the straight
	line program for the paths $(g_j)_{j=1,\ldots,f}$
	has space complexity $\tilde O(m.n)$ with the presented algorithm.
	But in our case, the output of the algorithm in Theorem
	\ref{thm:algofdom} yields a dual graph embedding with faces
	of size at most $3$. So that this step is quasi-linear. In
	general when $m$ is big, it may happen that
	there are $\tilde O(m)$ paths $\alpha_j$ in the same 
	face. But for $j<k$ with $\alpha_j$ contained in $\alpha_k$
	we can reuse the computation of $\alpha_j$ to obtain 
	$\alpha_k$ and get a complexity of $\tilde O(n)$.
	
	Building $\bar \gamma$ has complexity $\tilde O(n)$ as each
	edge of $T$ is crossed at most twice. This proves that
	the one-word presentation can be obtained in quasi-linear
	time.

	Given a one-word presentation, extracting $m$-handles
	has complexity $\tilde O(m.n)$. Indeed,
	using the notations of \ref{subsec:mhandles}, extracting a
	handle is just a matter of finding the tuple $(c,d)$,
	rewriting $w$ and writing $f$ and $h$ which is done in 
	$\tilde O(n)$ time and space. For the geometric presentation, this yields
	$\tilde O(n^2)$ time and space, using the straight line program
	representation for the canonical loops $(a_i,b_i)_{i=1,\ldots,g}$.
	\begin{rem}
		Using a direct word representation of $a_g$ would require
		exponential space as each $a_i$ contains $\prod_{j=1}^{i-1}[a_j,b_j]$.
	\end{rem}
\begin{lem}\label{lem:asymptotic}
	Let $\mu$ be the area of a fundamental polygon $P$ for $\Gamma$. 
	The co-area $\mu(\Gamma\backslash \h)=\mu$ of $\Gamma$ is linear
	in the number of edges of $P$.
\end{lem}
\begin{proof}
	Let $P$ be a fundamental polygon for $\Gamma$. 
	Its area is given by the formula
	\[\mu(P)=(n-2)\pi-\sum_{i=1}^n\alpha_i\]
	where the $(\alpha_i)$ are the interior angles at vertices of $P$. We
	will prove that at least $\frac{n}{6}$ of the angles are smaller than
	$\frac{2\pi}{3}$. As all angles are at most $\pi$ by convexity, this shows that
    $n = O(\mu(P))$.

	We first prove that no two consecutive angles can both be equal to $\pi$. 
	 Let $\alpha$
	be an interior angle of $P$ for consecutive edges $e$ and $\sigma_2(e)$, and
	$\beta$ the interior angle of $\sigma_2(e)$ and $\sigma_2(\sigma_2(e))$. 
	We have $\beta,\alpha\leq \pi$ by convexity.
	We also have $\alpha=\pi$ if and only if
	$\sigma_2(e)=g.e$ for some $g\in\Gamma$ of order $2$ by definition of edges of $P$. This can be rewritten
	as $\sigma_2(e)=\sigma_1(e)$. Let us suppose that $\alpha=\beta=\pi$. Then
	$\sigma_2(\sigma_2(e))=\sigma_1(\sigma_2(e))$ and $\sigma_2(e)=\sigma_1(e)$
	yielding $\sigma_2(\sigma_2(e))=e$
	which says that $P$ has at most $2$ edges but 
	as $P$ has non-zero area, it has at least $3$ edges so this is a
	contradiction. We now prove the result.

	Fix a vertex $v$ of $P$ with angle $\alpha<\pi$ and let $C=\partial P$. If
	$D_v$ is the set of darts of $\Gamma.C$ incident to $v$ and
	$A$ is the set of
	angles between adjacent darts of $D_v$ then $A$ is a subset
	of $\{\alpha_i\}_i$
	because $\Gamma$ acts conformally on $\h$. From the identity
	$\sum_{\beta\in A}\beta=2\pi$ and the fact that $\alpha$ is 
	in $A$ we must have
	$|A|\geq 3$ so
	that at least a third of the elements of $A$ are less than $2\pi/3$. This
    concludes the proof of the lemma.
\end{proof}
We have almost proved our main theorem:
\begin{thm}\label{thm:algopres}
	Let $\Gamma$ be a co-compact Fuchsian group with signature 
	$(g,\nu_1,\ldots, \nu_f)$ identified by a fundamental polygon and
	a side pairing. There exists an algorithm that
	computes a one-word 
	(resp. $m$-handles, resp. geometric) presentation of $\Gamma$
	in time $\tilde O(\mu+\sum_{i=1}^f\log(\nu_i))$ (resp. 
	$\tilde O(m.\mu+\sum_{i=1}^f\log(\nu_i))$, resp. $\tilde O(\mu^2+\sum_{i=1}^f\log(\nu_i))$).
\end{thm}
\begin{proof}
	The theorem follows from Lemma \ref{lem:asymptotic}, the complexity 
	analysis of Algorithm \ref{algo:mainslp} made in Section \ref{subsec:complexity1} and the following remark.
	We assume given a geometric representation of a fundamental
	polygon $P$. That is a representation for its edges in 
	$\h$ and a set of real numbers representing angles 
	for each successive edges. For each word $c=e_1\ldots e_k$
	of edges representing a vertex of $P$, and if the 
	angle between $e_ie_{i\mod k + 1}$ is $\alpha_i$,
	we have $\sum_{i=1}^k \alpha_i=2\pi/m$ for some integer $m$.
	Further, when 
	$c$ is elliptic and of order $\nu_i$, we have $m=\nu_i$.
	For each $e\in c$, assuming each $\alpha_i$
	is given with precision at least 
	$\frac{1}{\nu_i(\nu_i+1)}$, we can recover $\nu_i$ in time
	$O(\log(\nu_i^2))=O(\log(\nu_i))$. Now each elliptic
	element $\gamma_j$ of a special presentation corresponds
	uniquely to some $\nu_i$ and the association is given
	by the function $h$ computed in Algorithm 
	\ref{algo:dfsordering}.
\end{proof}

\begin{ex}[Geometric presentation for $\Gamma$ over $\Q(\sqrt8)$.]\label{ex:geomprescongsg13}
	Using the notations of example \ref{ex:congsg13} and figure \ref{fig:gis},
	following Algorithms \ref{algo:gis} and \ref{algo:gamma} we compute
	$\gamma_1=\bar{a}$, $\gamma_2=a\bar{b}\bar{e}$, $\gamma_3=h_3.(b\bar{c}d).\overline{h_3}$,
$\gamma_4=h_3.(\bar{d}ec).\bar{h_3}$ and
	$\gamma=\gamma_1.\gamma_2\gamma_3\gamma_4=\bar{c}.ec\bar{e}$.
	As $a$ is the only elliptic, from the signature of $\Gamma$ we obtain the 
	geometric presentation
    \[ \Gamma\simeq\langle
      c,e,\gamma_1,\gamma_2,\gamma_3,\gamma_4\mid[e,c^{-1}]\gamma_1\gamma_2\gamma_3\gamma_4=1,~\gamma_1^3=1\rangle. \]
\end{ex}

\section{Cohomology of Fuchsian groups}
\label{sec:cohomology}

Our goal is to compute the space~$M_{k+2}(\Gamma)$ of modular forms of
weight~$k+2$, where~$k\ge 0$ is an even integer; we refer to
\cite[Section~43.9]{quatbook} or~\cite[Chapter~2]{shimura} for definitions.
By the Eichler--Shimura isomorphism, this reduces to the computation of a
cohomology group; we refer to~\cite{greenbergvoightHilbert,dembelevoightHilbert}
for more details and focus on this computation.
For an integer~$k\ge 0$, we denote~$\Sym^k(\C^2)$ the $k$-th symmetric power of
the standard representation of~$\SL_2(\C)$ on~$\C^2$; when~$k$ is even this
action descends to~$\PSL_2(\C)$ and therefore makes sense as a representation of
any Fuchsian group.

\begin{defn}\label{def:cohom}
  Let~$\Gamma$ be a group and~$V$ a representation of~$\Gamma$.
  Define the \emph{space of $1$-cocycles}
    $Z^1(\Gamma,V) = \{\sigma\colon \Gamma \to V \mid \sigma(ab) =
    \sigma(a)+a\sigma(b) \text{ for all }a,b\in\Gamma\}$,
  the \emph{space of $1$-coboundaries}
    $B^1(\Gamma,V) = \{a \mapsto (a-1)v : v\in V\} \subset Z^1(\Gamma,V)$
  and the \emph{first cohomology group}
    $H^1(\Gamma,V) = Z^1(\Gamma,V) / B^1(\Gamma,V)$.
  Let~$e_1,\dots,e_n$ be generators of~$\Gamma$. Evaluation on the generators
  defines an injection
  $Z^1(\Gamma,V) \hookrightarrow V^n$.
  By \emph{computing $H^1(\Gamma,V)$}, we mean computing a subspace~$H$ of~$V^n$
  such that
  $Z^1(\Gamma,V) = B^1(\Gamma,V) \oplus H \subset V^n$.
\end{defn}

As long as we can express arbitrary elements of~$\Gamma$ as words in the
generators, this representation of~$H^1(\Gamma,V)$ is suitable for the
computation of Hecke operators
(see~\cite{greenbergvoightHilbert,dembelevoightHilbert} for details).
If the fundamental polygon is a Dirichlet domain,
this can be done as long as we can express the generators of the
special presentation in terms of the ones
corresponding to the fundamental polygon and conversely.
These expressions are obtained as straight-line programs by our
Algorithm~\ref{algo:mainslp} and Remark \ref{rem:invslp}, and it is clear that
cocycles can be evaluated on elements given by straight-line programs.
At the time of writing we have not yet implemented the Hecke action.

The main problem we want to solve is the following.

\begin{prob}
  Given~$k$ and~$\Gamma$, compute~$H^1(\Gamma,\Sym^k(\C^2))$.
\end{prob}

From an arbitrary presentation of total size~$O(n)$, we can apply linear algebra
directly to compute~$H^1(\Gamma,V)$ in time~$O((n\dim V)^\omega)$, but we will
reduce the dependence on~$n$ by using the topological methods of
Section~\ref{sec:topology}.


As a warm-up, we explain the easy case~$k=0$.
From a one word presentation of a Fuchsian group~$\Gamma$,
we can
efficiently compute
$H^1(\Gamma,\C) = \Hom(\Gamma,\C)$. Indeed, elliptic elements map to~$0$, and in the
remaining one-word relation,
the hyperbolic generators cancel modulo commutators
(see Section~\ref{subsec:fgge})
so there is no
additional constraint: the morphisms~$\sigma_i \colon (e_j \mapsto
\delta_{i=j}, \gamma_j \mapsto 0)$ form a basis of~$H^1(\Gamma,\C)$.
In this section we show that one can obtain similar efficiency
for~$k>0$
from a one-handle presentation.

\begin{nota}
  Let~$\Gamma$ be a group and~$V$ a representation of~$\Gamma$, and let~$a,b\in \Gamma$.
  Write~$\Phi_{a,b}\colon V^2 \to V$ the map defined by
  \[
    \Phi_{a,b}(v_1,v_2) = (1-aba^{-1})v_1 + a(1-ba^{-1}b^{-1})v_2.
  \]
  Write~$\Psi_{a,b}\colon V \to V^2$ the map defined by
    $\Psi_{a,b}(v) = ((a-1)v, (b-1)v)$.
\end{nota}

Throughout this section, 
  let $\Gamma$ be a co-compact Fuchsian group, let
  \[
    \langle a,b,e_3,\dots,e_{2g},\gamma_1,\dots,\gamma_f
    \mid [a,b]w, \gamma_1^{\nu_1},\dots \gamma_f^{\nu_f}\rangle
  \]
  be a one-handle
  presentation of~$\Gamma$, and
  let~$V$ be a finite dimensional representation
  of~$\Gamma$ over a field~$K$.
We now provide an explicit description of~$H^1(\Gamma,V)$ under mild hypotheses.

\begin{lem}\label{lem:Z1onehandle}
  Assume that~$\Phi_{a,b}$ is surjective.
  Let~$n=2g+f$. For each~$2<i\le n$, define the subspace~$N_i\subset V^n$ as
  follows. Let~$\gamma$ be the $i$-th generator, i.e. $\gamma=e_i$ if~$i\le 2g$
  and $\gamma = \gamma_{i-2g}$ otherwise. Define
  \begin{itemize}
    \item $V'=V$ if~$\gamma$ is hyperbolic, and
    \item $V' = \ker(1+\gamma+\dots+\gamma^{\nu-1})$ if~$\gamma$ is elliptic of
      order~$\nu$.
  \end{itemize}
  For each~$v'\in V'$, let~$(v_1,v_2)\in V^2$ be a solution
  of~$\Phi_{a,b}(v_1,v_2) = v''$ where~$v''$ is the sum of all terms of the
  following form, one for each occurrence of~$\gamma$ or~$\gamma^{-1}$ in~$w$:
  \begin{itemize}
    \item $w_1\cdot v'$ for every word decomposition~$w = w_1\gamma w_2$, and
    \item $-w_1\gamma^{-1}\cdot v'$ for every word decomposition~$w = w_1 \gamma^{-1} w_2$.
  \end{itemize}
  Define~$N_i$ to be the span of the vectors~$(v_1,v_2,0,\dots,0,v',0,\dots,0)$
  where~$v'$ ranges over a basis of~$V'$.

  Then
  \[
    Z^1(\Gamma,V) = (\ker\Phi_{a,b}\times 0^{n-2}) \oplus \bigoplus_{i=3}^n N_i
    \subset V^n.
  \]
\end{lem}
\begin{proof}
  Let~$\sigma\in Z^1(\Gamma,V)$.
  For an elliptic generator~$\gamma$, the cocycle condition evaluated on the relation~$\gamma^{\nu}$ gives
  $\sigma(\gamma) \in \ker(1+\gamma+\dots+\gamma^{\nu-1})$.
  The one-handle relation can be written
  $0 = \sigma([a,b]w) = \Phi_{a,b}(\sigma(a),\sigma(b)) + [a,b]\sigma(w)$.
  Since~$\Phi_{a,b}$ is surjective, this means that the~$\sigma(e_i)$ for~$i>2$
  and~$\sigma(\gamma_j)$ can take arbitrary values satisfying the elliptic
  relations, and that~$\sigma(a)$ and~$\sigma(b)$ are constrained only by the
  above equation. This proves that the claimed decomposition is correct.
\end{proof}

\begin{lem}\label{lem:H1onehandle}
  Assume~$g\ge 2$ and let~$c=e_3$ and~$d = e_4$.
  Assume that~$\Phi_{a,b}$ is surjective and~$\Psi_{c,d}$ is injective.
  Then
  \[
    Z^1(\Gamma,V) = B^1(\Gamma,V) \oplus (\ker\Phi_{a,b}\times 0^{n-2}) \oplus
    N_{3,4} \oplus \bigoplus_{i>4} N_i
    \subset V^n,
  \]
  where the~$N_i$ are as in Lemma~\ref{lem:Z1onehandle} and  $N_{3,4}$ is formed
  as~$N_3\oplus N_4$ in Lemma~\ref{lem:Z1onehandle} with~$V^2$ replaced by a complement of~$\im\Psi_{c,d}$ in~$V^2$.
\end{lem}
\begin{proof}
  First, we have~$N_{3,4}\subset Z^1(\Gamma,V)$ by construction.
  Moreover, the map~$\Psi_{c,d}$ is injective, therefore by restricting to the $3$rd and
  $4$th components and applying Lemma~\ref{lem:Z1onehandle} we see that $B^1(\Gamma,V)$ and
  the subspace described in the statement indeed decompose $Z^1(\Gamma,V)$ into
  a direct sum.
\end{proof}

\begin{lem}\label{lem:injsurj}
  Assume~$\dim V\ge 2$.
  Let~$a,b\in \Gamma$ be such that~$\langle a,b\rangle$ acts irreducibly on~$V$
  and assume that~$\dim V^a \le 1$ and~$\dim V^b \le 1$.
  Then~$\Phi_{a,b}$ is surjective and~$\Psi_{a,b}$ is injective.
\end{lem}
\begin{proof}
  Let~$W_1 = \im(1-aba^{-1})$ and~$W_2 = \im (a(1-ba^{-1}b^{-1}))$, so that we
  have the equality~$\im \Phi_{a,b} = W_1 + W_2$.
  We have~$\codim W_1 = \dim \ker(1-aba^{-1}) = \dim V^b \le 1$
  and similarly~$\codim W_2 \le 1$.

  Assume for contradiction that~$W_1+W_2\neq V$.
  Then by the previous codimension inequalities we must have~$W_1 = W_2 \neq V$,
  and~$W_1\neq 0$ since~$\dim V\ge 2$.
  Now~$a^{-1}W_1$ is stable under~$b$ and~$a^{-1}W_2 = \im (1-ba^{-1}b^{-1})$ is
  stable under~$ba^{-1}b^{-1}$.
  Therefore~$a^{-1}W_1=a^{-1}W_2$ is stable under
    $\langle b, ba^{-1}b^{-1}\rangle
    = \langle b, a^{-1}\rangle
    = \langle a,b \rangle$,
  but this contradicts the irreducibility hypothesis. Therefore~$W_1+W_2=V$.

  We have~$\ker\Psi_{a,b} = V^a \cap V^b$, on which~$\langle a,b\rangle$ acts
  trivially, and which has dimension at most~$1$. Since $\dim V\ge 2$, by
  the irreducibility hypothesis we must have~$V^a \cap V^b = 0$, so
  that~$\Psi_{a,b}$ is injective.
\end{proof}

We now show that the hypotheses on~$\Phi_{a,b}$ and~$\Psi_{c,d}$ are always
satisfied.

\begin{lem}\label{lem:zariskidense}
  Let~$a,b\in\SL_2(\C)$ be non-commuting hyperbolic elements such that the group~$\langle
  a,b\rangle$ contains no parabolic element.
  Then~$\langle a,b\rangle$ is Zariski-dense in~$\SL_2/\C$.
\end{lem}
\begin{proof}
  The maximal algebraic subgroups of~$\SL_2$ are finite subgroups, Borel
  subgroups and normalisers of tori. Let~$G$ be the Zariski closure of~$\langle
  a,b\rangle$. Then
    $G$ cannot be contained in a finite group since $a$ and $b$ have
      infinite order;
    if~$G$ is contained in a Borel subgroup, then the commutator~$[a,b]$ is parabolic or
      the identity, but that is contrary to our hypothesis;
    if~$G$ is contained in the normaliser~$N$ of a torus~$T$, then every
      element of~$N(\C)\setminus T(\C)$ is elliptic so~$G\subset T$, but that
      contradicts the fact that~$a$ and~$b$ do not commute.
  Therefore~$G = \SL_2$ as claimed.
\end{proof}

\begin{lem}\label{lem:irred}
  Let~$\Gamma$ be a co-compact Fuchsian group and~$a,b\in\Gamma$ be noncommuting
  hyperbolic elements, and let~$k\ge 0$ be even.
  Then~$\langle a,b\rangle$ acts irreducibly on~$\Sym^k(\C^2)$.
\end{lem}
\begin{proof}
  By Lemma~\ref{lem:zariskidense}, $\langle a,b\rangle$ is Zariski-dense in $\SL_2$.
  But~$\Sym^k(\C^2)$ is irreducible as an algebraic representation of~$\SL_2$,
  so it is irreducible as a representation of~$\langle a,b\rangle$.
\end{proof}

We are finally in position to prove the main result of this section.

\begin{prop}\label{prop:computeH1}
  Assume~$g\ge 2$.
  Then from the one-handle presentation of~$\Gamma$ and an explicit representation~$\rho\colon \Gamma\to
  \GL(V)$ over a field~$K\subset \C$ such that~$V\otimes_K \C \cong
  \Sym^k(\C^2)$, one can compute a basis of~$H^1(\Gamma,V)$
  in~$O((k+1)^\omega (n+\sum_j \log \nu_j))$
  operations in~$K$, where~$n = 2g+f$.
\end{prop}
\begin{proof}
  We may assume~$k>0$.
  Let~$c=e_3$ and~$d = e_4$.
  By Lemma~\ref{lem:irred}, the subgroups~$\langle a,b\rangle$
  and~$\langle c,d\rangle$ act irreducibly on~$V$: indeed~$[a,b]w$ is a minimal
  relation involving~$a,b,c$ or~$d$, so~$[a,b]\neq 1$ and~$[c,d]\neq 1$.

  We can therefore apply Lemma~\ref{lem:injsurj} to conclude that~$\Phi_{a,b}$
  is surjective and~$\Psi_{c,d}$ is injective: indeed, since~$a,b,c,d$ are
  hyperbolic, their fixed-point space in~$\Sym^k(\C^2)$ has dimension~$1$, so
  the same is true in~$V$.

  We therefore have an explicit description of~$H^1(\Gamma,V)$ by
  Lemma~\ref{lem:H1onehandle}. We explain how to compute the described basis
  in time~$O(k^\omega (n+\sum_j \log \nu_j))$. Note that~$\dim_K V = \dim_\C
  \Sym^k(\C^2) = k+1$.
  \begin{enumerate}
    \item compute~$\ker\Phi_{a,b}$ in time~$O(k^\omega)$;
    \item compute a complement of~$\im\Psi_{c,d}\subset V^2$ in time~$O(k^\omega)$;
    \item compute all the elements of~$\GL(V)$ corresponding to the action of
      the prefixes of~$[a,b]w$ in time~$O(k^\omega n)$;
    \item for each~$1\le j\le f$:
      \begin{itemize}
        \item Compute~$\rho(1)+\rho(\gamma_j) + \dots + \rho(\gamma_j)^{\nu_j-1}$
          in time~$O(k^\omega\log\nu_j)$ by binary powering;
        \item Compute~$\ker(\rho(1)+\rho(\gamma_j) + \dots +
          \rho(\gamma_j)^{\nu_j-1})$
          in time~$O(k^\omega)$;
      \end{itemize}
    \item Compute~$N_{3,4}$ in time~$O(k^\omega)$;
    \item For each $4<i\le n$: compute~$N_i$ in time~$O(k^\omega)$.
  \end{enumerate}
  This proves the proposition.
\end{proof}

\begin{rem}
  The explicit description in Lemma~\ref{lem:H1onehandle} is compatible with
  Shimura's dimension formula~\cite[Theorem~2.23]{shimura} via the
  Eichler--Shimura isomorphism
  \[
    H^1(\Gamma,\Sym^k(\C^2)) \cong M_{k+2}(\Gamma)^2.
  \]
\end{rem}

\begin{rem}
  Proposition~\ref{prop:computeH1} is trivial in the non-co-compact case, since one
  can remove the long relation and one parabolic generator: $\Gamma$ is a free
  product of cyclic groups.
\end{rem}

In order to obtain Theorem~\ref{thm:specialpres} and
Theorem~\ref{thm:introcomputeH1} for congruence arithmetic Fuchsian groups from
Theorem~\ref{thm:algopres} and Proposition~\ref{prop:computeH1}, note the
following standard facts (see for instance~\cite{genus0shim}):
\begin{itemize}
  \item
	When $\Gamma$ is arithmetic, coming from a quaternion algebra $A$ over a
	field $F$, each elliptic element
	corresponds to an embedding of a cyclotomic field $K$ into $A$.
    For each $i$, this gives the bound
    $\phi(\nu_i)\leq 2[F:\Q]$ where~$\phi$ denotes Euler's totient function, so
    that~$\log \nu_i = O(\log\, [F:\Q])$. Moreover, we have~$[F:\Q] = O(\mu)$
    by the structure of arithmetic Fuchsian groups, their volume formula and
    Odlyzko's bound~\cite{odlyzko}, so that~$\log\nu_i = O(\log\mu)$.
  \item By a theorem of Selberg and Zograf~\cite{zograf}, there is only a finite
    number of congruence arithmetic Fuchsian groups of genus at most~$1$. We can
    therefore compute the cohomology using direct linear algebra for them and
    apply Proposition~\ref{prop:computeH1} when~$g\ge 2$.
\end{itemize}

\section{Implementation}\label{sec:implem}

We implemented the topological algorithms in PARI/GP~\cite{PARI2}; the code is
available at \url{https://github.com/rayanebait/Cohomology-of-Shimura-curves},
and uses Rickards's PARI code to compute fundamental domains~\cite{rickards}.
We measured the time to compute the cohomology group~$H^1(\Gamma,\C)$ in the
following example, using our approach and using Magma's current implementation.
Let $F = \Q(t)$ where~$t^3-3t-1=0$ and let~$A$ be the quaternion
algebra~$\left(\frac{-3,4t}{F}\right)$.
We fix a maximal order~$\Or$ in~$A$ and for various prime
ideals~$\p$ we let~$\Gamma =
\Or_0(\p)^\times/\Z_F^\times$
where~$\Or_0(\p)\subset \Or$ is an Eichler order of
level~$\p$.
The results are displayed in Figure~\ref{fig:time} in logarithmic scale.
More precisely, we first measure the time to compute a fundamental domain
for~$\Gamma$ using \texttt{X = afuchinit(A,ord,1)} from Rickards's package, and report it as
\textsf{fdom} in Figure~\ref{fig:time}. We then measure the time to compute a
one word presentation for~$\Gamma$ using Theorem~\ref{thm:algopres}
(\texttt{afuch\_presentation(X,"oneword")} in our implementation); as explained
in Section~\ref{sec:cohomology} this
gives an explicit basis of~$H^1(\Gamma,\C)$ without extra computation.
In Magma, we use the commands \texttt{M := HilbertCuspForms(F,N :
QuaternionOrder := OO); Dimension(M : UseFormula := false)}. This computes the
dimension of~$H^1(\Gamma,\C)$, but by forbidding the use of the dimension
formula we force the computation of an actual basis for the cohomology space.
However, by varying the level and providing a fixed order, we prevent the
computation of new fundamental domains, so that the measured time only
corresponds to cohomology computation.
The algorithm is described in~\cite[Section 6]{greenbergvoightHilbert}: it uses
direct linear algebra via Shapiro's lemma.
For the same reason, we could compute a fundamental polygon by a
topological method using the covering map; however these measured timings shows
how our algorithm would behave in a setting where the discriminant varies
instead of the level.
The slope of the best linear fit is $4.18$ for \textsf{magma}, $1.54$ for
\textsf{fdom} (conjectured to be~$2$ asymptotically~\cite[Section 5]{rickards}) and $1.00$
for \textsf{oneword}.

\begin{figure}
	\centering
    \includegraphics[width=0.7\textwidth]{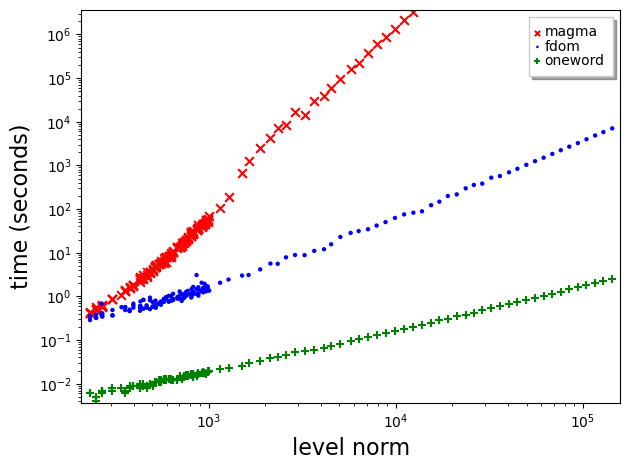}
	\caption{Computing~$H^1(\Gamma,\C)$}
	\label{fig:time}
\end{figure}

%
%


%

\providecommand{\bysame}{\leavevmode\hbox to3em{\hrulefill}\thinspace}
\providecommand{\MR}{\relax\ifhmode\unskip\space\fi MR }
\providecommand{\MRhref}[2]{%
  \href{http://www.ams.org/mathscinet-getitem?mr=#1}{#2}
}
\providecommand{\href}[2]{#2}

\end{document}